\documentclass{amsart}
\input xy
\xyoption{all}
\newtheorem{theorem}{Theorem}[section]
\newtheorem{proposition}[theorem]{Proposition}
\newtheorem{corollary}[theorem]{Corollary}
\newtheorem{conjecture}[theorem]{Conjecture}
\newtheorem{problem}[theorem]{Problem}

\newtheorem{lemma}[theorem]{Lemma}
\newcommand{\Hom}{\operatorname{Hom}}

\theoremstyle{definition}
\newtheorem{definition}[theorem]{Definition}
\newtheorem{remark}[theorem]{Remark}
\newtheorem{example}[theorem]{Example}

\newcommand{\GHilb}{\Z^{\widehat{G}}}
\newcommand{\Poly}{Q}

\newcommand{\R}{{\mathbb R}}
\newcommand{\C}{{\mathbb C}}
\newcommand{\GL}{\operatorname{GL}}
\newcommand{\Q}{{\mathbb Q}}
\newcommand{\id}{\operatorname{id}}
\newcommand{\PolyMat}{\operatorname{\it PolyMat}}
\newcommand{\Mat}{\operatorname{\it Mat}}
\newcommand{\QSym}{\operatorname{\it QSym}}
\newcommand{\NSym}{\operatorname{\it NSym}}
\newcommand{\im}{\operatorname{im}}
\newcommand{\reg}{\operatorname{reg}}

\newcommand{\Proj}{{\mathbb P}}
\newcommand{\Z}{{\mathbb Z}}
\newcommand{\pd}{\operatorname{pd}}

\newcommand{\Tor}{\operatorname{Tor}}
\newcommand{\N}{{\mathbb N}}
\newcommand{\Sym}{\operatorname{{\it Sym}}}
\newcommand{\Symc}{\operatorname{\overline{\it Sym}}}
\newcommand{\rk}{\operatorname{rk}}
\newcommand{\Power}{\operatorname{Pow}}
\title[(quasi-)symmetric functions of polymatroids]{Symmetric and Quasi-Symmetric Functions associated to Polymatroids}
\author{Harm Derksen}
\thanks{The author is partially supported by the  NSF, grant DMS 0349019.}%

\begin{document}
\maketitle
\begin{abstract}
To every subspace arrangement  ${\bf X}$ we will
associate symmetric functions ${\mathcal P}[{\bf X}]$ and ${\mathcal H}[{\bf X}]$.
These symmetric functions encode the Hilbert series and the minimal projective
resolution of the product ideal associated to the subspace arrangement. They
can be defined for
discrete polymatroids as well.
 The invariant ${\mathcal H}[{\bf X}]$ specializes to the Tutte polynomial ${\mathcal T}[{\bf X}]$.
 Billera, Jia and Reiner recently introduced a quasi-symmetric function ${\mathcal F}[{\bf X}]$
 (for matroids) which behaves valuatively with respect to
 matroid base polytope decompositions. We will define a quasi-symmetric function 
 ${\mathcal G}[{\bf X}]$ for polymatroids which has this property as well. Moreover, ${\mathcal G}[{\bf X}]$
 specializes to ${\mathcal P}[{\bf X}]$, ${\mathcal H}[{\bf X}]$, ${\mathcal T}[{\bf X}]$
 and ${\mathcal F}[{\bf X}]$.
  
\end{abstract}
\setcounter{tocdepth}{1}
\tableofcontents

\section{Introduction}
\subsection{Combinatorial invariants.}
Let $X$ be a set with $d$ elements. Suppose that $V_x, x\in X$ are subspaces
of an $n$-dimensional vector space. Then ${\mathcal A}=\bigcup_{x\in X}V_x$
is called a {\em subspace arrangement}. Let $\Power(X)$ be the set of all
subsets of $X$.
The {\em rank function} 
$\rk:\Power(X)\to \N:=\{0,1,2,\dots\}$
is defined 
by 
$$
\rk(A)=\dim V-\dim\textstyle \bigcap_{i\in A} V_i
$$
for all subsets $A\subseteq X$.

Surprisingly, many topological invariants of the complement $V\setminus {\mathcal A}$
of subspace arrangements are {\em combinatorial}, i.e., they can be expressed
in terms of $n:=\dim V$ and the rank function. 
For example, Zaslavsky (see~\cite{Zas}) proved that number of regions in the
complement of a real hyperplane arrangement
is equal to
$$
(-1)^n\chi(-1)=\sum_{A\subseteq X}(-1)^{\rk(A)+|A|},
$$
where $\chi(q)$ is the {\em characteristic polynomial} of the hyperplane arrangement
defined by
$$
\chi(q)=\sum_{A\subseteq X}q^{n-\rk(A)}(-1)^{|A|}.
$$
For {\em complex hyperplane arrangements}, the cohomology ring
$H^{\star}(V\setminus {\mathcal A})$ is isomorphic to the {\em Orlik-Solomon algebra} (see~\cite{OS}),
which is defined explicitly in terms of the rank function.
For {\em arbitrary real subspace arrangements}, the topological Betti numbers
of the complement $V\setminus {\mathcal A}$ are expressed
in terms of the rank function using the {\em Goresky-MacPherson formula} (see~\cite{GM}). 

One may wonder whether various {\em algebraic} objects associated to a subspace
arrangements are combinatorial invariants. Let $K$ be a base field of characteristic 0,
and denote the coordinate ring of $V$ by $K[V]$. Terao defined
the module of derivations $D({\mathcal A})$ along a hyperplane arrangement
${\mathcal A}$ (see~\cite{Terao}). An arrangement is called {\em free} if $D({\mathcal A})$ is
a free $K[V]$-module. Terao has conjectured that ``freeness''
is a combinatorial property, i.e., whether $D({\mathcal A})$ is free is
 determined by its rank function.  Terao showed that free arrangements
have the property that their characteristic polynomial factors into linear polynomials (see~\cite{Terao}).
One should point out that for example
the {\em Hilbert series} of the module $D({\mathcal A})$ is {\em not} a combinatorial invariant.

In a recent paper, the author found an {\em algebraic} object which is 
a combinatorial invariant for subspace arrangements. 
Let $J_x\subseteq K[V]$ be the vanishing ideal
of $V_x\subseteq V$ and let $J=\prod_{x\in X}J_x$ be the product ideal.
The author showed in \cite{Derksen} that the Hilbert series $H(J,t)$ of $J$ is a combinatorial invariant.
For {\em hyperplane} arrangements the Hilbert series of $J$ is always
equal to $t^d/(1-t)^n$ and is therefore not an interesting invariant.
Let $W$ be an arbitrary vector space and denote its dual by $W^\star$.
We can tensor all the spaces with $W^\star$.
So let $J_x(W)\subseteq K[V\otimes W^\star]$ be the vanishing ideal of the subspace 
$V_x\otimes W^\star$ of
$V\otimes W^\star$ and $J(W)=\prod_{x\in A}J_x(W)$.
Then the Hilbert series $H(J(W),t)$ {\em is} an interesting invariant, even for
hyperplane arrangements. Moreover, since we have an action of $\GL(W)$ on all
the rings and ideals involved, we can define a $\GL(W)$-equivariant
Hilbert series which is a more refined invariant for  subspace arrangement.

\subsection{Symmetric functions}
The ring of symmetric functions is spanned by the Schur symmetric functions $s_\lambda$ where
$\lambda$ runs over all partitions. Let ${\bf X}=(X,\rk)$ where $\rk$ is the
rank function coming from a subspace arrangement $\bigcup_{x\in X} V_x\subseteq V$.
In Section~\ref{sec:2.2},we will define a symmetric function ${\mathcal P}[{\bf X}]$
using a recursive formula (see Definition~\ref{def3}). We  define another symmetric function $
{\mathcal H}[{\bf X}]={\mathcal H}[{\bf X}](q,t)$
with coefficients in $\Z[q,t]$ by
\begin{equation}
{\mathcal H}[{\bf X}](q,t)=\sum_{A\subseteq X}{\mathcal P}[{\bf X}\mid_A]q^{\rk(A)}t^{|A|}.
\end{equation}
Here ${\bf X}\mid_A=(A,\rk\mid_{A})$  can be viewed as the rank function
of the sub-arrangement $\bigcup_{x\in A}V_x\subseteq V$.
The definitions of ${\mathcal P}[{\bf X}]$ and ${\mathcal H}[{\bf X}](q,t)$ make sense even if the rank function $\rk$
does not come from a subspace arrangement. Therefore, these symmetric functions can also be defined for
{\em polymatroids}.
The symmetric function ${\mathcal H}[{\bf X}](q,t)$ essentially encodes Hilbert series of $J$
and the $\GL(W)$-equivariant Hilbert series of $J(W)$.
Also, the {\em minimal free resolutions} of $J$ and  $J(W)$ can be expressed in terms of ${\mathcal H}[{\bf X}](q,t)$.
The symmetric functions behave nicely with respect direct sums of polymatroids, namely
\begin{eqnarray}
{\mathcal P}[{\bf X}\oplus {\bf Y}] &=& {\mathcal P}[{\bf X}]\cdot {\mathcal P}[{\bf Y}]\\
{\mathcal H}[{\bf X}\oplus {\bf Y}](q,t) &=& {\mathcal H}[{\bf X}](q,t)\cdot {\mathcal H}[{\bf Y}](q,t)
\end{eqnarray}
(see Proposition~\ref{prop:mult}).
The {\em Tutte polynomial}  is defined by
\begin{equation}
{\mathcal T}[{\bf X}](x,y)=\sum_{A\subseteq X}(x-1)^{\rk(X)-\rk(A)}(y-1)^{|A|-\rk(A)}.
\end{equation}
The Tutte polynomial was introduced in \cite{Tutte} and generalized
to matroids in \cite{Brylawski} and  \cite{Crapo}. It has the multiplicative
property and it behaves well under matroid duality (see~(\ref{eqTutteDual})). It specializes
to the characteristic polynomial, namely 
$$\chi(q)=q^{n-\rk(X)}{\mathcal T}[{\bf X}](1-q,0).$$
The
coefficients of ${\mathcal T}[{\bf X}](x,y)$ as a polynomial in $x$ and $y$ 
have combinatorial interpretations and
are nonnegative. 
The  invariants ${\mathcal H}[{\bf X}](q,t)$ specializes to the Tutte polynomial.
The functions ${\mathcal P}[{\bf X}]$ and ${\mathcal H}[{\bf X}](q,t)$ do {\em not} seem to behave nicely
under matroid duality.
If the polymatroid ${\bf X}$ is realizable as a subspace arrangement in characteristic 0, then
the coefficients of ${\mathcal P}[{\bf X}]$, ${\mathcal H}[{\bf X}](q,t)$ and some of
their specializations have homological interpretations. Therefore, the coefficients
of these functions satisfy certain non-negativity conditions.

Brylawski defined a graph invariant in \cite{Brylawski2} which he called the {\em polychromate}. 
Sarmiento \cite{Sarmiento} proved that the polychromate is equivalent
to the {\em U-polynomial} studied by Noble and Welch \cite{NW}.
The polychromate and the U-polynomial specialize to Stanley's {\em chromatic symmetric polynomial}  \cite{Stanley}. 
There are graphs whose graphical matroids are the same, that can be distinguised
by the Stanley symmetric function. This means that the Stanley symmetric function,
the polychromatic, and the U-polynomial {\em cannot} be viewed as
 invariants of matroids. 
 
 Inspired by these graph invariants,  Billera, Jia and Reiner defined a {\em quasi-symmetric}
function which {\em is} an invariant for matroids (see~\cite{BJR}). This
invariant  will be discussed later.
\subsection{Polarized Schur functions}
Let us denote the Schur functor corresponding to the partition $\lambda$ by $S_{\lambda}$.
Suppose our base field $K$ has characteristic $0$, $Z$ is a finite dimensional $K$-vector space,
and $Z_1,\dots,Z_d\subseteq Z$ are subspaces.
For a partition $\lambda$ with $|\lambda|=d$ we will define a subspace
$$
S_{\lambda}(Z_1,Z_2,\dots,Z_d)\subseteq S_{\lambda}(Z)
$$
as the subspace spanned by the all $\pi(z_1\otimes \cdots\otimes z_d)$
where $z_i\in Z_i$ for all $i$ and
$$
\pi:\underbrace{Z\otimes Z\otimes \cdots\otimes Z}_d\to S_{\lambda}(Z)
$$
is a $\GL(Z)$-equivariant linear map.

The space $S_{\lambda}(Z_1,\dots,Z_d)$ has various interesting properties
which will be discussed in Section~\ref{sec:5}. For example 
$$S_{\lambda}(\underbrace{Z,Z,\dots,Z}_d)=S_{\lambda}(Z).
$$
Also, permuting the spaces $Z_1,\dots,Z_d$ does not change the subspace $S_{\lambda}(Z_1,\dots,Z_d)$.
Let $V=Z^\star$ be the dual space, and define $V_i=Z_i^{\perp}$ to be the subspace of $V$ orthogonal to $Z_i$.
Consider the subspace arrangement ${\mathcal A}=V_1\cup\cdots \cup V_d\subseteq V$.
Then the dimension of $S_{\lambda}(Z_1,\dots,Z_d)$ can be expressed 
in terms of ${\mathcal H}[{\mathcal A}](q,t)$.
This implies, that the dimension of $S_{\lambda}(Z_1,\dots,Z_d)$
is determined by the numbers
$$
\dim \textstyle \sum_{i\in A} Z_i,\quad A\subseteq \{1,2,\dots,d\}.
$$

\subsection{Quasi-symmetric functions}
Billera, Jia and Reiner defined a quasi-symmetric function ${\mathcal F}[{\bf X}]$ for any matroid ${\bf X}$ in 
\cite{BJR}.
This invariant behaves nicely with respect to direct sums of matroids,
matroid duality. There is also a very natural definition of this invariant in terms of the
combinatorial Hopf algebras studied in \cite{ABS} (see Section~\ref{sec:7.4}). In \cite{BJR} it was proved that
this quasi-symmetric function behaves valuatively with respect to matroid polytope decompositions,
so it can be a useful tool for studying such decompositions.
The quasi-symmetric  ${\mathcal F}[{\bf X}]$ does not specialize to ${\mathcal H}[{\bf X}](q,t)$
because  ${\mathcal F}[{\bf X}]$ cannot distinguish between a loop or an isthmus,
and ${\mathcal H}[{\bf X}](q,t)$ can.
We will show that ${\mathcal F}[{\bf X}]$ {\em does} specialize to ${\mathcal P}[{\bf X}]$. To prove this,
we introduce another quasi-symmetric function ${\mathcal G}[{\bf X}]$ which should be of
interest on its own right. 
First of all, we will choose a convenient basis $\{U_{r}\}$ of the ring of quasi-symmetric functions
where $r$ runs over all finite sequences of nonnegative integers.
A complete chain is a sequence 
$$
\underline{X}:\emptyset=X_0\subset X_1\subset \cdots \subset X_d=X
$$
such that $X_i$ has $i$ elements for all $i$. The rank vector of this chain $\underline{X}$ is defined by
$$
r(\underline{X})=(\rk(X_1)-\rk(X_0),\dots,\rk(X_d)-\rk(X_{d-1})).
$$
Now we define
$$
{\mathcal G}[{\bf X}]=\sum_{\underline{X}}U_{r(\underline{X})}
$$
where $\underline{X}$ runs over all $d!$ maximal chains in $X$. 
We will show that ${\mathcal G}[{\bf X}]$ behaves nicely with respect to direct sums and matroid duality.
It defines a Hopf algebra homomorphism from
the Hopf algebra of polymatroids to the Hopf algebra of quasi-symmetric functions.
But unlike ${\mathcal F}[{\bf X}]$, it can distinguish between a loop 
and an isthmus.
Moreover, ${\mathcal G}[{\bf X}]$
 specializes to the Billera-Jia-Reiner quasi-symmetric function ${\mathcal F}[{\bf X}]$ as well
as to ${\mathcal H}[{\bf X}](q,t)$.
We will also show that ${\mathcal G}[{\bf X}]$ has
the valuative property with respect to polymatroid polytope decompositions in Section~\ref{sec:8}.
We question whether ${\mathcal G}[{\bf X}]$ might be universal with this property.
\subsection*{Acknowledgement}
The author would like to thank Nathan Reading, Frank Sottile, David Speyer for
inspiring discussions and helpful suggestions.

\section{Symmetric functions associated to polymatroids}
In this section we will define the invariants ${\mathcal H}[{\bf X}](q,t)$ and ${\mathcal P}[{\bf X}]$.
\subsection{Discrete polymatroids}

\begin{definition}
A (discrete) {\em polymatroid} is a pair ${\bf X}:=(X,\rk)$ where $X$ is a finite set, and 
$\rk:\Power(X)\to \N=\{0,1,2,\dots\}$
is a function satisfying
\begin{enumerate}
\item $\rk(\emptyset)=0$;
\item $\rk(A)\leq \rk(B)$ if $A\subseteq B$ (nondecreasing);
\item $\rk(A\cup B)+\rk(A\cap B)\leq \rk(A)+\rk(B)$ (submodular).
\end{enumerate}
\end{definition}
If ${\bf X}=(X,\rk)$ is a polymatroid, and $A\subseteq X$ is a subset, then we restrict ${\bf X}$ to $A$
to get a polymatroid ${\bf X}\mid_A:=(A,\rk\mid_A)$. If $A^{\rm c}=X\setminus A$
is the complement, then the {\em deletion} of $A$ in ${\bf X}$ is the
polymatroid ${\bf X}\setminus A:={\bf X}\mid_{A^{\rm c}}=(A^{\rm c},\rk\mid_{A^{\rm c}})$.
The polymatroid ${\bf X}/A:=(A^{\rm c},\rk_{X/A})$ is defined by
$$
\rk_{X/A}(B)=\rk(A\cup B)-\rk(A)
$$
for all $B\subseteq  A^{\rm c}$. We call ${\bf X}/A$ the {\em contraction} of $A$ in ${\bf X}$.

Two polymatroids ${\bf X}=(X,\rk_X)$ and ${\bf Y}=(Y,\rk_Y)$ are {\em isomorphic} if there exists a bijection $\varphi:X\to Y$
such that $\rk_Y\circ\varphi=\rk_X$. A polymatroid ${\bf X}=(X,\rk_X)$ is a 
{\em matroid} if $\rk_X(\{x\})\in \{0,1\}$ for all $x\in X$.
If ${\bf X}=(X,\rk_X)$ is a matroid, then its dual is ${\bf X}^{\vee}:=(X,\rk_X^{\vee})$ where $\rk_X^{\vee}$ is defined by
$$
\rk_X^{\vee}(A):=|A|-\rk_X(X)+\rk_X(X\setminus A)
$$
for all $A\subseteq X$. The Tutte polynomial behaves nicely with respect to matroid duality:
\begin{equation}\label{eqTutteDual}
{\mathcal T}[{\bf X}^{\vee}](x,y)={\mathcal T}[{\bf X}](y,x).
\end{equation}
There is also a formula
expressing ${\mathcal F}[{\bf X}^\vee]$ in terms of ${\mathcal F}[{\bf X}]$ (see~\cite{BJR}).
\begin{definition}
If ${\bf X}=(X,\rk_X)$ and ${\bf Y}=(Y,\rk_Y)$ are polymatroids, then we define their
direct sum by
$$
{\bf X}\oplus {\bf Y}:=(X\sqcup Y,\rk_{X\sqcup Y})
$$
where $X\sqcup Y$ is the disjoint union of $X$ and $Y$ and $\rk_{X \sqcup Y}:X\sqcup Y\to\N$ is defined by
$$\rk_{X\sqcup Y}(A\cup B):=\rk_X(A)+\rk_Y(B)
$$
for all $A\subseteq X$, $B\subseteq Y$.
\end{definition}
The Tutte polynomial satisfies the multiplicative property
\begin{eqnarray}
{\mathcal T}[{\bf X}\oplus {\bf Y}] &=& {\mathcal T}[{\bf X}]\cdot {\mathcal T}[{\bf Y}].
\end{eqnarray}

\subsection{The ring of symmetric functions}
Let 
$$\Sym:=\Z[e_1,e_2,e_3,\dots]\subset \Z[x_1,x_2,x_3,\dots]$$
be the ring of symmetric functions in infinitely many variables, where 
$$
e_k:=\sum_{i_1<i_2<\cdots<i_k}x_{i_1}x_{i_2}\cdots x_{i_k}
$$
is the $k$-th elementary symmetric function.
The monomials in $e_1,e_2,\dots$ form a $\Z$-basis of $\Sym$. 
A partition of $n$ is a tuple $\lambda=(\lambda_1,\lambda_2,\dots,\lambda_r)$ 
of positive integers with
$\lambda_1\geq \cdots \geq \lambda_r\geq 1$ and $|\lambda|:=\lambda_1+\cdots+\lambda_r$ equal to $n$.
Another basis of $\Sym$ is given by the
Schur symmetric functions $s_{\lambda}$ where $\lambda$ runs over all partitions.
For standard results for symmetric functions, we refer to the book \cite{McD}.
The natural grading of $\Z[x_1,x_2,x_3,\dots]$ induces a grading on $\Sym$.
In this grading $e_k$ has degree $k$ and $s_{\lambda}$ has degree $|\lambda|$.
Let 
$$
\Symc=\Z[[e_1,e_2,e_3,\dots]]
$$
be the set of power series in $e_1,e_2,\dots$.
Define
$$
\sigma=1+s_1+s_2+s_3+\cdots\in \Symc.
$$
The inverse is given by
\begin{equation}\label{eq1}
\sigma^{-1}=1-e_1+e_2-e_3+\cdots=1-s_1+s_{11}-s_{111}+\cdots.
\end{equation}
\subsection{The definitions of ${\mathcal P}[{\bf X}]$ and ${\mathcal H}[{\bf X}](q,t)$}\label{sec:2.2}

\begin{definition}\label{def3}
For every polymatroid ${\bf X}=(X,\rk)$ we define a symmetric polynomial ${\mathcal P}[{\bf X}]\in \Sym$ by induction
as follows. If $X=\emptyset$, then
${\mathcal P}[{\bf X}]=1$. If $X\neq\emptyset$, then we may assume that ${\mathcal P}[{\bf X}\mid_A]$ has
been defined for all {\em proper} subsets $A\subset X$.
We define
\begin{equation}\label{eq2}
{\mathcal P}[{\bf X}]=u_0+u_1+\cdots+u_{|X|-1}
\end{equation}
where $u_i\in \Sym$ is homogeneous of degree $i$ for all $i$ such that
\begin{equation}\label{eq3}
\sum_{i=0}^\infty u_i=-\sum_{A\subset X} {\mathcal P}[{\bf X}\mid_A]\sigma^{\rk(X)-\rk(A)}(-1)^{|X|-|A|}.
\end{equation}
Here $A$ runs over all proper subsets of $X$.
\end{definition}
\begin{definition}\label{def4}
For every polymatroid ${\bf X}=(X,\rk)$ we define a symmetric polynomial 
$${\mathcal H}[{\bf X}](q,t)\in \Sym[q,t]=\Z[q,t]\otimes_{\Z} \Sym$$ 
by
\begin{equation}\label{eq4}
{\mathcal H}[{\bf X}](q,t)=\sum_{A\subseteq X} {\mathcal P}[{\bf X}\mid_A]q^{\rk(A)}t^{|A|}.
\end{equation}
\end{definition}
The coefficient of $t^{|X|}$ in ${\mathcal H}[{\bf X}](q,t)$ is
$q^{\rk(X)}{\mathcal P}[{\bf X}]$.
\begin{remark}\label{rem5}
If we evaluate (\ref{eq4}) at $q=\sigma^{-1}$ and $t=-1$, then 
we obtain
$$
{\mathcal H}[{\bf X}](\sigma^{-1},-1)=\sum_{A\subseteq X} {\mathcal P}[{\bf X}\mid_A]\sigma^{-\rk(A)}(-1)^{|A|}\in \Symc.$$
From  (\ref{eq2}) and (\ref{eq3}) it follows that ${\mathcal H}[{\bf X}](\sigma^{-1},-1)$
vanishes in degree  $<d=|X|$.
\end{remark}
\begin{proposition}[multiplicative property]\label{prop:mult}
For polymatroids ${\bf X}=(X,\rk_X)$ and ${\bf Y}=(Y,\rk_Y)$ we have
\begin{equation}\label{eq6}
{\mathcal P}[{\bf X}\oplus {\bf Y}]={\mathcal P}[{\bf X}]\cdot{\mathcal P}[{\bf Y}].
\end{equation}
and
\begin{equation}\label{eq7}
{\mathcal H}[{\bf X}\oplus {\bf Y}](q,t)={\mathcal H}[{\bf X}](q,t)\cdot{\mathcal H}[{\bf Y}](q,t).
\end{equation}
\end{proposition}
\begin{proof}
We prove the proposition by induction on $|X|+|Y|$. The case where $X=Y=\emptyset$ is clear. 
So let us assume that $|X|+|Y|>0$. We may assume that
$$
{\mathcal P}[{\bf X}\mid_A\oplus  {\bf Y}\mid_B]={\mathcal P}[{\bf X}\mid_A]\cdot {\mathcal P}[{\bf Y}\mid_B]
$$
for all subsets $A\subseteq X$ and $B\subseteq Y$ such that $A\neq X$ or $B\neq Y$.
\begin{multline}\label{eq8}
{\mathcal H}[{\bf X}\oplus {\bf Y}](q,t)=\sum_{C\subseteq X\sqcup Y}
{\mathcal P}[({\bf X}\oplus {\bf Y})\mid_{C}]q^{\rk_{X\sqcup Y}(C)}t^{|C|}=\\
=
\sum_{A\subseteq X}\sum_{B\subseteq Y}
{\mathcal P}[{\bf X}\mid_A\oplus  {\bf Y}\mid_B]q^{\rk_X(A)+\rk_Y(B)}t^{|A|+|B|}=\\
=
\sum_{A\subseteq X}
{\mathcal P}[{\bf X}\mid_A]q^{\rk_X(A)}t^{|A|}\cdot
\sum_{B\subseteq Y}
{\mathcal P}[{\bf Y}\mid_B]q^{\rk_Y(B)}t^{|B|}+\\
+\big({\mathcal P}[{\bf X}\oplus {\bf Y}]-{\mathcal P}[{\bf X}]{\mathcal P}[{\bf Y}]\big)
q^{\rk_X(X)+\rk_Y(Y)}t^{|X|+|Y|}=\\
{\mathcal H}[{\bf X}](q,t)\cdot{\mathcal H}[{\bf Y}](q,t)+
\big({\mathcal P}[{\bf X}\oplus  {\bf Y}]-{\mathcal P}[{\bf X}]{\mathcal P}[{\bf Y}]\big)
q^{\rk_X(X)+\rk_Y(Y)}t^{|X|+|Y|}
\end{multline}

If we substitute $q=\sigma^{-1}$ and $t=-1$ we get
\begin{multline*}
{\mathcal H}[{\bf X}\oplus {\bf Y}](\sigma^{-1},-1)-{\mathcal H}[{\bf X}](\sigma^{-1},-1)\cdot {\mathcal H}[{\bf Y}](\sigma^{-1},-1)=\\
=(-1)^{|X|+|Y|}
\big({\mathcal P}[{\bf X}\oplus {\bf Y}]-{\mathcal P}[{\bf X}]\cdot {\mathcal P}[{\bf Y}]\big)\sigma^{-\rk_X(X)-
\rk_Y(Y)}
\end{multline*}
The left-hand side has no terms in degree $<|X|+|Y|$ by Remark~\ref{rem5}
and 
$${\mathcal P}[{\bf X}\oplus  {\bf Y}]-{\mathcal P}[{\bf X}]-{\mathcal P}[{\bf Y}]$$
is a symmetric polynomial of degree $<|X|+|Y|$.
It follows that
$$
{\mathcal P}[{\bf X}\oplus {\bf Y}]={\mathcal P}[{\bf X}]\cdot {\mathcal P}[{\bf Y}].
$$
From (\ref{eq8}) follows that
$$
{\mathcal H}[{\bf X}\oplus {\bf Y}](q,t)={\mathcal H}[{\bf X}](q,t)\cdot {\mathcal H}[{\bf Y}](q,t).
$$
\end{proof}
The Tutte polynomial is closely related to the {\em  rank generating function}
$$
{\mathcal R}[{\bf X}](q,t)=\sum_{A\subseteq X}q^{\rk(A)}t^{|A|}
$$
We have 
$$
(x-1)^{\rk(X)}{\mathcal R}[{\bf X}]((y-1)^{-1}(x-1)^{-1},(y-1))={\mathcal T}[{\bf X}](x,y),
$$
so the Tutte polynomial is completely determined by the rank generating function
and vice versa.  The rank generating function makes sense for polymatroids,
not just matroids. The Tutte invariant may not be a polynomial
for polymatroids, because we could have $\rk(A)>|A|$ for some subset $A\subseteq X$.
Define
$$
\Theta:\Sym\to \Q
$$
by
$$
\Theta(s_{\lambda})=\left\{
\begin{array}{ll}
1 &\mbox{if $\lambda=()$;}\\
0 & \mbox{otherwise.}\end{array}\right.
$$
Using base extension, we also get a $\Q(q,t)$-linear map
$$
\Sym\otimes_{\Q}\Q(q,t)\to\Q(q,t)
$$
which we also will denote by $\Theta$.
It is straightforward to prove by induction on $|X|$ that $\Theta({\mathcal P}[{\bf X}])=1$.
\begin{corollary}
We have
$$
\Theta({\mathcal H}[{\bf X}](q,t))=\sum_{A\subseteq X}q^{\rk(A)}t^{|A|}={\mathcal R}[{\bf X}](q,t).
$$
So ${\mathcal H}[{\bf X}](q,t)$ specializes to the rank generating function
and the Tutte polynomial.
\end{corollary}
\section{Examples}
\begin{example}
Let ${\bf 0}=(\{v\},\rk_{\bf 0})$ be the loop matroid, and ${\bf 1}=(\{v\},\rk_{\bf 1})$
be the co-loop matroid defined by
$$
\rk_{\bf 0}(v)=0\mbox{ and } rk_{\bf 1}(v)=1.
$$
Then we have  $P[{\bf 0}]=P[{\bf 1}]=1$, 
${\mathcal H}[{\bf 0}]=1+t$, ${\mathcal H}[{\bf 1}]=1+qt$, ${\mathcal G}[{\bf 0}]=U_{(0)}$
and ${\mathcal G}[{\bf 1}]=U_{(1)}$.
\end{example}
An important class of matroids is the class of graphical matroids. Suppose that $\Gamma=(Y,X,\phi)$
where $Y$ is the set of vertices, $X$ is the set of edges,
and $\phi:X\to \Power(Y)$ is a map such that $\phi(e)$ is the set of endpoints of the edge $e$.
So $\phi(e)$ has $1$ or $2$ elements for all $e\in X$.
Let $V=K^n$, and denote the coordinate functions by $x_1,\dots,x_n$.
To each vertex $e\in X$, with $\phi(e)=\{i,j\}$
we can associate a subspace $V_e\subseteq V$ defined by
$x_i=x_j$. So $V_e$ is a hyperplane unless $e$ is a loop (i.e., $i=j$), in which
case $V_e=V$.
For $A\subseteq X$, we define $V_A=\bigcap_{a\in A}V_a$.
We define a rank function by
$$
\rk(A)=\dim V-\dim V_A,\qquad A\subseteq X.
$$
Now ${\bf X}=(X,\rk)$ is a matroid.
\begin{example}
Suppose $(Y,X,\phi)$ is an $m$-gon. 
$$
m=6:
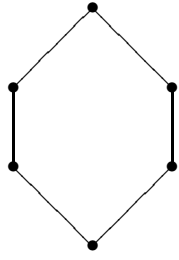
\xymatrix{
& *=0{\bullet}\ar@{-}[rd]\ar@{-}[ld] &\\
*=0{\bullet}\ar@{-}[d] & & *=0{\bullet}\ar@{-}[d]\\
*=0{\bullet} & & *=0{\bullet}\\
& *=0{\bullet}\ar@{-}[lu]\ar@{-}[ru]}
$$
Then we have
$$
{\mathcal T}[{\bf X}](x,y)=y+x+x^2+\cdots+x^{m-1}\
$$
$$
{\mathcal P}[{\bf X}]=1-s_1+s_{11}-\cdots+(-1)^{m-1}s_{1^{m-1}}.
$$
$$
{\mathcal H}[{\bf X}](q,t)=(1+qt)^m-(qt)^m+q^{m-1}t^{m}{\mathcal P}[{\bf X}],
$$
$$
{\mathcal G}[{\bf X}]=m!U_{(1,1,\dots,1,0)}
$$
\end{example}
\begin{example}\label{exmedges}
Suppose that $(Y,X,\phi)$ is the graph with $2$ vertices and $m$ edges between them.
$$
m=5:
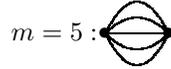
\xymatrix{
*=0{\bullet}\ar@/^1pc/@{-}[r] 
\ar@/^.5pc/@{-}[r]
\ar@{-}[r]
\ar@/_.5pc/@{-}[r]
\ar@/_1pc/@{-}[r]& *=0{\bullet}}
$$
Then we have
$$
{\mathcal T}[{\bf X}](x,y)=x+y+y^2+\cdots+y^{m-1}
$$
\begin{equation}\label{eq:Pmarr}
{\mathcal P}[{\bf X}]=\textstyle 1-{m-1\choose 1}s_1+{m-1\choose 2}s_{2}-\cdots+(-1)^{m-1}{m-1\choose m-1}s_{m-1}.
\end{equation}
\begin{equation}\label{eq:Hmarr}
{\mathcal H}[{\bf X}](q,t)=1+q\sum_{i=1}^m {m\choose i}t^i\left(\sum_{j=0}^{i-1}(-1)^j{i-1\choose j}s_j\right).
\end{equation}
Here, we use the convention $s_0=1$.
To prove the formulas (\ref{eq:Pmarr}) and (\ref{eq:Hmarr}) it suffices to show that
the right-hand side of (\ref{eq:Hmarr}) vanishes in degree $<m$ if we substitute $q=\sigma^{-1}$ and $t=-1$. If we make these substitutions, we get (using the combinatorial
identity  \cite[\S1.2.6, (33)]{Knuth})
\begin{multline}
1+\sigma^{-1}\sum_{i=1}^m {m\choose i}(-1)^i\left(\sum_{j=0}^{i-1}(-1)^j{i-1\choose j}s_j\right)=\\
1+\sigma^{-1}\sum_{j=0}^{m-1}s_j\sum_{i=j+1}^{m}(-1)^{i+j}{m\choose i}{i-1\choose j}=\\
1+\sigma^{-1}\sum_{j=0}^{m-1}s_j\left((-1)^{j+1}{-1\choose j}+
\sum_{i=0}^{m}(-1)^{i+j}{m\choose i}{i-1\choose j}\right)=
1-\sigma^{-1}\sum_{j=0}^{m-1}s_j.
\end{multline}
This vanishes in degree $<m$ because $\sigma=1+s_1+s_2+\cdots$.

We also have 
$$
{\mathcal G}[{\bf X}]=m!U_{(1,0,0,\dots,0)}.
$$

\end{example}
The following example appeared in \cite{Brylawski2}, and was pointed out to the author by Nathan Reading.
\begin{example}
The Gray graphs
$$
G_1=
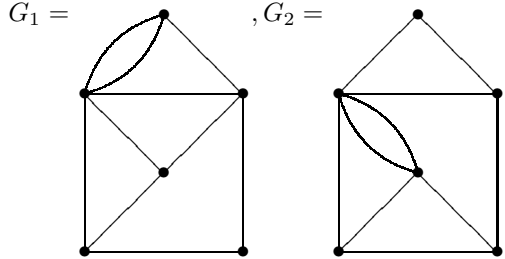
\xymatrix{
& *=0{\bullet}\ar@{-}@/^/[ld]\ar@{-}@/_/[ld]\ar@{-}[rd] & \\
*=0{\bullet}\ar@{-}[rr] \ar@{-}[rd]\ar@{-}[dd]& & *=0{\bullet}\ar@{-}[ld]\ar@{-}[dd]\\
& *=0{\bullet}\ar@{-}[ld] & \\
*=0{\bullet}\ar@{-}[rr] &  & *=0{\bullet}
},
G_2=\xymatrix{
& *=0{\bullet}\ar@{-}[rd]\ar@{-}[ld] & \\
*=0{\bullet}\ar@{-}[rr] \ar@{-}@/_/[rd]\ar@{-}@/^/[rd]\ar@{-}[dd]& & *=0{\bullet}\ar@{-}[dd]\\
& *=0{\bullet}\ar@{-}[ld]\ar@{-}[rd] & \\
*=0{\bullet}\ar@{-}[rr] &  & *=0{\bullet}
}$$
have the same Tutte polynomial, namely
\begin{multline*}
{\mathcal T}[G_1](x,y)={\mathcal T}[G_2](x,y)=y^5+4y^4+xy^4+x^2y^3+6xy^3+7y^3+x^3y^2+6y^2+6x^2y^2+\\+13xy^2+
10xy+
x^4y+13x^2y+6x^3y+2y+2x+7x^3+x^5+4x^4+6x^2.
\end{multline*}
However, the coefficients of $s_{2,2,2}$ in ${\mathcal P}[G_1]$ and ${\mathcal P}[G_2]$
are  $56$ and $55$ respectively.
\end{example}
The examples below appeared in the survey of Brylawski and Oxley in \cite[pp. 197]{White}, and were
also featured in \cite{BJR}.
\begin{example}
Consider $6$ points in $\Proj^2=\Proj^2(\C)$ according to the diagram below 
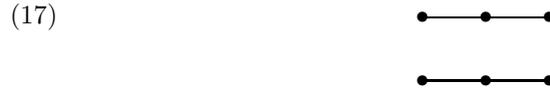
\begin{equation}\label{eq:diagram1}
\xymatrix{
*=0{\bullet}\ar@{-}[r] & *=0{\bullet}\ar@{-}[r] & *=0{\bullet}\\
*=0{\bullet}\ar@{-}[r] & *=0{\bullet}\ar@{-}[r] & *=0{\bullet}}
\end{equation}
Here $3$ or more points are collinear if and only if they lie on a line segment in the diagram.
Dualizing gives us $6$ projective lines in $\Proj^2$ which can be viewed as 6 hyperplanes in $\C^3$.

Denote the matroid associated with this arrangement by ${\bf X}$.
Consider $6$ points in $\Proj^2$ according to the diagram below 
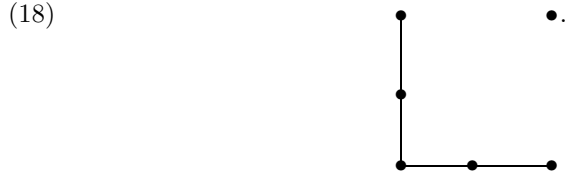
\begin{equation}
\xymatrix{
*=0{\bullet}\ar@{-}[d] &  & *=0{\bullet}\\
*=0{\bullet}\ar@{-}[d] & &\\
*=0{\bullet}\ar@{-}[r] & *=0{\bullet}\ar@{-}[r] & *=0{\bullet}}.
\end{equation}
Again, dualizing gives a hyperplane arrangement in $\C^3$.
Denote the matroid associated with this arrangement by ${\bf Y}$.

Then ${\bf X}$ and ${\bf Y}$ give nonisomorphic matroids, but
they have the same Tutte polynomial and the same Billera-Jia-Reiner quasi-symmetric function (see~\cite{BJR}).
Moreover,
\begin{multline*}
{\mathcal P}[{\bf X}]={\mathcal P}[{\bf Y}]=
1-3s_1+3s_2+6s_{1,1}-s_3-8s_{2,1}-8s_{1,1,1}+3s_{3,1}+6s_{2,2}+11s_{2,1,1}\\
-3s_{3,2}-4s_{3,1,1}-3s_{2,2,1},
\end{multline*}
$$
{\mathcal H}[{\bf X}](q,t)={\mathcal H}[{\bf Y}](q,t),
$$
and
$$
{\mathcal G}[{\bf X}]={\mathcal G}[{\bf Y}]=72U_{(1,1,0,1,0,0)}+648U_{(1,1,1,0,0,0)}.
$$
The last equation can easily be computed by hand as follows.
There are $6!$ ways of labeling the points in diagram~(\ref{eq:diagram1}) by $p_1,p_2,p_3,p_4,p_5,p_6$.
If $p_1,p_2,p_3$ are not colinear, then the labeling gives the rank sequence $(1,1,0,1,0,0)$,
because $p_1$ spans a  subspace of dimension $1$ in $\C^3$, $p_1$ and $p_2$ span a subspace
of dimension $1+1$, $p_1,p_2,p_3$ span a subspace of dimension $1+1+0$, $p_1,p_2,p_3,p_4$
span a subspace of dimension $1+1+0+1$, etc.
There are $2\cdot 3!^2=72$ ways of choosing a labeling such that $p_1,p_2,p_3$ are colinear.
All other $720-72=648$ labelings, give the rank sequence $(1,1,1,0,0,0)$.
A similar reasoning can be used to compute ${\mathcal G}[{\bf Y}]$.
\end{example}
\begin{example}
Let ${\bf X}$ be the matroid corresponding to the hyperplane arrangement dual to the point arrangement
of the following diagram
$$\xymatrix{
*=0{\bullet\bullet}\ar@{-}[d] & & & \\
*=0{\bullet}\ar@{-}[d] & & & \\
*=0{\bullet}\ar@{-}[r] & *=0{\bullet}\ar@{-}[r] &*=0{\bullet}\ar@{-}[r] & *=0{\bullet}}.
$$
Let ${\bf Y}$ be the matroid corresponding to the hyperplane arrangement dual to the point arrangement
of the following diagram
$$
\xymatrix{
*=0{\bullet}\ar@{-}[ddd]\ar@{-}[rrdddd] &&&&&&\\
&&&&&&\\
&&&&&&\\
*=0{\bullet}\ar@{-}[ddd]\ar@{-}[rrd] &&&&&&\\
& & *=0{\bullet}\ar@{-}[rdd]\ar@{-}[rrrrdd] &&&&\\
&&&&&&\\
*=0{\bullet\bullet}\ar@{-}[rrr] & & & *=0{\bullet}\ar@{-}[rrr] & & & *=0{\bullet}}.
$$
The Tutte polynomial is the same for ${\bf X}$ and ${\bf Y}$.
The Billera-Jia-Reiner quasi-symmetric function {\em does} distinguish the arrangements.
We have
\begin{multline*}
{\mathcal P}[{\bf X}]=1-4s_1+6s_2+9s_{1,1}-4s_3-17s_{2,1}-10s_{1,1,1}+
s_4+12s_{3,1}+13s_{2,2}+17s_{2,1,1}\\
-3s_{4,1}-10s_{3,2}-10s_{3,1,1}-8s_{2,2,1}+
2s_{4,2}+2s_{4,1,1}+2s_{3,3}+3s_{3,2,1}+s_{2,2,2}.
\end{multline*}
and
\begin{multline*}
{\mathcal P}[{\bf Y}]=1-4s_1+6s_2+9s_{1,1}-4s_3-17s_{2,1}-10s_{1,1,1}+
s_4+12s_{3,1}+14s_{2,2}+17s_{2,1,1}\\
-3s_{4,1}-12s_{3,2}-10s_{3,1,1}-10s_{2,2,1}+
3s_{4,2}+2s_{4,1,1}+2s_{3,3}+4s_{3,2,1}+s_{2,2,2}.
\end{multline*}
We also have
\begin{multline*}
{\mathcal G}[{\bf X}]=3456 U_{(1,1,1,0,0,0,0)}+1080U_{(1,1,0,1,0,0,0)}+
264U_{(1,1,0,0,1,0,0)}+\\
+216U_{(1,0,1,1,0,0,0)}+24U_{(1,0,1,0,1,0,0)}.
\end{multline*}
and
\begin{multline*}
{\mathcal G}[{\bf Y}]=3456 U_{(1,1,1,0,0,0,0)}+1104U_{(1,1,0,1,0,0,0)}+
240 U_{(1,1,0,0,1,0,0)}+\\
+192U_{(1,0,1,1,0,0,0)}+48U_{(1,0,1,0,1,0,0)}.
\end{multline*}
So the invariants ${\mathcal H},{\mathcal P}$ and ${\mathcal G}$ distinguish these two matroids as well.
\end{example}

\section{Ideals and regularity}
\subsection{Equivariant free resolutions}
Let $K$ be a field, and $V$ be an $n$-dimensional $K$-vector space.
For any partition $\lambda$, $S_\lambda$ denotes its corresponding Schur functor.
In particular, $S_dV$ is the $d$-th symmetric power of $V$, and $S_{1^d}V=S_{1,\dots,1}V$ is the $d$-th exterior power.
Let $R=K[V]$ be the ring of polynomial functions on $V$. The space
$R_d$ of polynomial functions of degree $d$ can be identified with
$S_d(Z)$, where $Z=V^\star$ is the dual space of $V$. 
Also, the ring $R=\bigoplus_{d=0}^\infty R_d$ can be
identified with the symmetric algebra
$S(Z):=\bigoplus_{d=0}^\infty S_d(Z)$ on $Z=V^\star$. 
By choosing a basis in $V$ and a dual basis $\{x_1,\dots,x_n\}$ in
$V^\star$ we may identify $R$ with the polynomial ring
$K[x_1,\dots,x_n]$. Let ${\mathfrak m}=\bigoplus_{d=1}^\infty
R_d=(x_1,\dots,x_n)$ be the maximal homogeneous ideal of $R$.

Suppose that $M$ is a finitely generated graded $R$-module. Its
minimal resolution can be constructed as follows. First define
$D_0:=M$ and $E_0=D_0/{\mathfrak m}D_0$. Then $E_0$ is a finite
dimensional, graded vector space. The homogeneous quotient map
$\psi_0:D_0\to E_0$ has a homogeneous linear section $\phi_0:E_0\to D_0$
(which does not need to be an $R$-module homomorphism)
such that $\psi_0\circ\phi_0=\id$. We can extend $\phi_0$ to a
$R$-module homomorphism $\phi_0:R\otimes_K E_0\to D_0$ in a unique
way. The tensor product $R\otimes_K E_0$ has a natural grading as a
tensor product of two graded vector spaces, and $\phi_0$ is
homogeneous with respect to this grading. We inductively define
$D_i,E_i,\psi_i,\phi_i$ as follows. Define $D_i$ as the kernel of
$\phi_{i-1}:R\otimes E_{i-1}\to D_{i-1}$. We set
$E_i=D_{i}/{\mathfrak m}D_{i}$. Let $\phi_i:E_i\to D_i$ be a
homogeneous linear section to the homogeneous quotient map $\psi_i:D_i\to
E_i$. We can extend $\phi_i$ to an $R$-module homomorphism
$\phi_i:R\otimes E_i\to D_i$. By Hilbert's Syzygy theorem (see \cite{Hilbert}
and \cite[Corollary 19.7]{Eisenbud}, we get
that $D_i=0$ for $i>n$. We end up with the minimal free resolution
$$
0\to R\otimes E_n\to R\otimes E_{n-1}\to \cdots R \otimes E_0\to
M\to 0.
$$
Here $E_i$ can be naturally identified with $\Tor_j(M,K)$.

For a group $G$ and sets $X$ and $Y$ on which $G$ acts, we say
that a map $\phi:X\to Y$ is $G$-equivariant if it respects the action, i.e.,
$\phi(g\cdot x)=g\cdot \phi(x)$ for all $x\in X$ and $g\in G$.
Suppose that $G$ is a linearly reductive linear algebraic group and
$V$ is a representation of $G$. Assume that $G$ also acts on the
finitely generated graded $R$-module $M=\bigoplus_d M_d$ such the
multiplication $R\times M\to M$ is $G$-equivariant, and $M_d$ is a
representation of $G$ for every $d$. By the definition of linear
reductivity, we can choose the sections $\phi_i:E_i\to K_i$ to be
$G$-equivariant. So by induction we see that $G$ acts regularly on
$D_0,E_0,D_1,E_1,D_2,E_2,\dots$. Also, by induction one can show
that the structure of $D_i$ as a $G$-equivariant graded $R$-module,
and $E_i$ as graded representation of $G$ do not depend on the
choices of the $G$-equivariant sections $\phi_i$. We conclude
that
$E_i\cong \Tor_i(M,K)$ has a well-defined structure as a graded
$G$-module.

\subsection{Castelnuovo-Mumford regularity}
For a finite dimensional graded $K$-vector space $W=\bigoplus_{d\in
\Z} W_d$ we define 
$$\deg(W):=\max\{i\mid W_i\neq 0\}.
$$
 If $W=\{0\}$
then we define $\deg(W)=-\infty$. A finitely generated graded
$R$-module $M$ is called $s$-regular if $\deg(\Tor^i(M,K))\leq s+i$
for all $i$. The {\it Castelnuovo-Mumford regularity\/} $\reg(M)$ of
$M$ is the smallest integer $s$ such that $M$ is $s$-regular.
See \cite[\S20.5]{Eisenbud} for more on Castelnuovo-Mumford regularity.

\subsection{Product ideals and regularity bounds}
Suppose that $V_x$, $x\in X$ are subspaces of $V$ for some finite set $X$ with $d$ elements.
Assume that $X=\{1,2,\dots,d\}$.
Let $J_x\subseteq K[V]=S(Z)$ be the vanishing ideal of $V_x$.
The ideal $J_x$ is generated by the subspace $Z_x=V_x^{\perp}\subseteq Z=V^\star$ of all linear functions vanishing 
on $V_x$. For every subset $A\subseteq X$, we define $J_A:=\prod_{x\in X} J_x$, and let $J=J_X$.
A  crucial result we need is:
\begin{theorem}[Conca and Herzog,\cite{CH}]\label{theoCH}
The Castelnuovo-Mumford regularity of $J$ is equal to $d$.
\end{theorem}
We define
\begin{equation}\label{eqCk}
C_k=\bigoplus_{|A|=k}J_A.
\end{equation}
Following \cite[Chapter IV]{Sidman} we construct a complex
\begin{equation}\label{eq:Scomplex}
0\to C_{d}\to C_{d-1}\to \cdots \to C_0\to 0.
\end{equation}
The map $\partial_k:C_k\to C_{k-1}$ can
be written as $\partial_k=\sum_{A,B}\partial_{k}^{A,B}$, where
$$
\partial_k^{A,B}:J_A\to J_B
$$
Suppose that $A=\{i_1,i_2,\dots,i_k\}$ with $i_1<i_2<\cdots < i_k$,
then we define
$$
\partial_k^{A,B}:=
\left\{
\begin{array}{ll}
0 & \mbox{if $B\not\subseteq A$;}\\
(-1)^r\id & \mbox{if
$B=\{i_1,\dots,i_{r-1},i_{r+1},\dots,i_k\}$.}\end{array}\right.
$$
The homology of the complex is denoted by
$$
H_k=\ker \partial_k/\im \partial_{k+1}.
$$
\begin{remark}\label{rem8}
Since $\partial_d$ is injective, we have that $H_d=0$.
\end{remark}

\begin{proposition}[\cite{Sidman}]
If $V_X:=\bigcap_{x\in X}V_x=(0)$, then the homogeneous
maximal ideal ${\mathfrak m}$ kills all homology, i.e.,
${\mathfrak m}H_i=0$ for all $i$.
\end{proposition}
The following result is Corollary 20.19 in \cite{Eisenbud}.
\begin{lemma}\label{lemmaABC}
If $A,B,C$ are finitely generated graded modules, and
$$
0\to A\to B\to C\to 0
$$
is exact, then
\begin{enumerate}
\item $\reg(A)\leq \max\{\reg(B),\reg(C)+1\}$;
\item $\reg(B)\leq \max\{\reg(A),\reg(C)\}$;
\item $\reg(C)\leq \max\{\reg(A)-1,\reg(B)\}$.
\end{enumerate}
\end{lemma}

\begin{proposition}\label{prop9}
Suppose that $V_X=\bigcap_{x\in X} V_x=(0)$.
Then $H_k$ is concentrated at degree $k$ (and in particular, it is finite dimensional).
\end{proposition}
\begin{proof}
We have $\reg(C_i)\leq i$ by Theorem~\ref{theoCH}.
Let $Z_i$ and $B_i$ be the kernel, respectively, the cokernel of $\partial_i$.

First, we prove that
\begin{equation}\label{eqHB}
\reg(H_i)\leq \reg(B_i)-1
\end{equation}
for $i=0,1,\dots,d-1$. 
Since ${\mathfrak m}H_i=0$, $H_i$ is just equal to a number of copies of $K$ in
various degrees. From the Koszul resolution follows that
$$
\deg(\Tor_j(H_i,K))=\deg(H_i)+j
$$
for $j=0,1,2,\dots,n$, hence $\reg(H_i)=\deg(H_i)$.
The exact sequence
\begin{equation}\label{eqBZH}
0\to B_i\to Z_i\to H_i\to 0
\end{equation}
gives rise to a long exact $\Tor$ sequence 
$$
0\to \Tor_n(B_i,K)\to\Tor_n(Z_i,K)\to \Tor_n(H_i,K)\to \Tor_{n-1}(B_i,K)\to \cdots
$$
Since $Z_i$ is a submodule of a free module, its projective dimension is $\leq n-1$ and $\Tor_n(Z_i,K)=0$.
Therefore 
$$\deg(\Tor_{n-1}(B_i,K))\geq \deg(\Tor_n(H_i,K))=\reg(H_i)+n.$$
It follows that
$$
\reg(B_i)+n-1\geq\deg(\Tor_{n-1}(B_i,K))\geq \reg(H_i)+n.
$$
This proves (\ref{eqHB}).

From (\ref{eqBZH}) and Lemma~\ref{lemmaABC} follows that
\begin{equation}\label{eqZB}
\reg(Z_i)\leq \max\{\reg(B_i),\reg(H_i)\}=\reg(B_i)
\end{equation}

By induction on $i$ we will show that $\reg(B_{d-i})\leq d-i+1$, $\reg(Z_{d-i})\leq d-i+1$ and $\reg(H_{d-i})\leq {d-i}$.
For $i=1$ we have $\reg(B_{d-1})=\reg(C_{d})=d$, $\reg(Z_{d-1})\leq d$ by (\ref{eqZB}) and $\reg(H_{d-1})\leq d-1$ by (\ref{eqHB}).

Suppose that $i>1$. We may assume by induction that $Z_{d-i+1}$ is $(d-i+2)$-regular.
From the exact sequence
$$
0\to Z_{d-i+1}\to C_{d-i+1}\to B_{d-i}\to 0
$$
follows that 
$$
\reg(B_{d-i})\leq \max\{\reg(Z_{d-i+1})-1,\reg(C_{d-i+1})\}\leq d-i+1
$$
by Lemma~\ref{lemmaABC}.
Now we have $\reg(Z_{d-i})\leq d-i+1$ by (\ref{eqZB}) and $\reg(H_{d-i})\leq d-i$ by (\ref{eqHB}).

\end{proof}

Suppose that $G$ is a linearly reductive group and let $\widehat{G}$ denote
the set of isomorphism classes of irreducible representations of $G$.
Let $\GHilb$ be the set of maps $\widehat{G}\to \Z$.
Elements of $\GHilb$ may be thought of as $G$-Hilbert series.
If $M$ is a $G$-module such that every irreducible representation
appears only finitely many times, then we define
$$\langle M\rangle=\langle M\rangle_G\in \GHilb.$$ 
For every irreducible represention $U$ of $G$, $\langle M\rangle(U)$
is the multiplicity of $U$ in $M$.
\begin{lemma}\label{lemEulerChar}
Suppose that $G$ acts on $Z$ such that every irreducible representation of $G$ appears only finitely many times in $S(Z)$.
Then we have
\begin{equation}\label{eqAlt}
\sum_{A\subset X}(-1)^{|A|}\langle J_A\rangle=
\sum_{i=0}^d (-1)^i \langle C_i\rangle=\sum_{i=0}^{d-1}(-1)^i\langle H_i\rangle
\end{equation}
\end{lemma}
\begin{proof}
The first equality follows from the definition \ref{eqCk}. 
For every $i$ we have exact sequences
$$
0\to Z_i\to C_i\to \im B_{i-1}\to 0
$$
and
$$
0\to B_i\to Z_i \to H_i\to 0.
$$
So we have
\begin{multline}
 \sum_{i} (-1)^i \langle C_i\rangle =
 \sum_{i} (-1)^i\langle Z_i\rangle+\sum_i(-1)^i\langle B_{i-1} rangle=\\
=
 \sum_i (-1)^i\langle Z_i\rangle-\sum_{i}(-1)^i\langle B_i\rangle=
 \sum_i (-1)^i\langle H_i\rangle.
\end{multline}
 \end{proof}

\section{Realizable polymatroids}
\subsection{The tensor trick}
Let us fix a field $K$.
\begin{definition}
A {\em arrangement realization} of a polymatroid ${\bf X}=(X,\rk)$ over $K$ is a finite dimensional $K$-vector space $V$ together 
with a collection of subspaces $V_x$, $x\in X$
such that
$$
\rk(A)=\dim V-\dim V_A
$$
for every $A\subseteq X$, where 
$$V_A=\bigcap_{x\in X}V_x.$$
\end{definition}
Let ${\bf X}=(X,\rk)$ be a polymatroid and set $d=|X|$.
From now on, assume that $K$ is a field of characteristic 0.
Suppose that $V$ is an $n$-dimensional $K$-vector space and
 $V_x$, $x\in X$ is a collection of subspaces that form a realization of ${\bf X}$.
 
Let $W$ be another $K$-vector space and let $R(W):=K[V\otimes W^\star]$ be the ring of polynomial functions on $V\otimes W^\star=\Hom(W,V)$.
Note that $\GL(W)$ acts regularly on $K[V\otimes W^\star]$.
Let $J_x(W)\subseteq R(W)$ be the vanishing ideal of $V_x\otimes W^\star\subseteq V\otimes W^\star$.
For a subset $A\subseteq X$ we define
$$
J_A(W)=\prod_{x\in A}J_x(W)
$$
and we set $J(W):=J_X(W)$.
Define
$$
C_i(W):=\bigoplus_{A\subseteq X\atop |A|=i}J_{A}(W).
$$
As in (\ref{eq:Scomplex}), we have a complex
\begin{equation}\label{eq:ScomplexW}
0\to C_d(W)\to C_{d-1}(W)\to \cdots \to C_1(W)\to C_0(W)\to 0.
\end{equation}
Let $H_i(W)$ be the $i$-th homology group. 
By Lemma~\ref{lemEulerChar}, we have
\begin{equation}\label{eq10.2}
\sum_{i=0}^{d-1}(-1)^i\langle H_i(W)\rangle=\sum_{i=0}^d(-1)^i\langle C_i(W)\rangle=
\sum_{A\subseteq X}(-1)^{|A|}\langle J_A(W)\rangle
\end{equation}
If $f=\sum_{\lambda}a_{\lambda}s_{\lambda}\in \Z[[e_1,e_2,\dots]]$, then we define
$$
f\star W=\sum a_{\lambda}\langle S_{\lambda}(W)\rangle.
$$
For example, we have
$$
\sigma\star W=(s_0+s_1+s_2+s_3+\cdots)\star W=\sum_{i=0}^\infty \langle S_i(W)\rangle=\langle S(W)\rangle.
$$
If $f,g\in \Z[[e_1,e_2,\dots]]$, then
$$
(f\cdot g)\star W=(f\star W)\otimes (g\star W).
$$
\subsection{Product ideals and the invariants ${\mathcal P}[{\bf X}]$, ${\mathcal H}[{\bf X}](q,t)$}
\begin{theorem}\label{theo8}
We have
\begin{equation}\label{eq10}
\big(\sigma^{n-\rk(X)}{\mathcal P}[{\bf X}]\big)\star W=\sum_{A\subseteq X}(-1)^{|A|}\langle J_A(W)\rangle
\end{equation}
and
\begin{equation}\label{eq10.1}
\big(\sigma^n{\mathcal H}[X](\sigma^{-1},-1)\big)\star W=\langle J(W)\rangle.
\end{equation}

\end{theorem}
\begin{proof}
We prove the statement by induction on $d=|X|$. If $X=\emptyset$, then
${\mathcal P}[{\bf X}]=1$ and
$$
\sigma^n\star W=\langle S(W)^{\otimes n}\rangle=\langle S(W\otimes V^\star)\rangle=\langle K[V\otimes W^\star]\rangle=\langle R(W)\rangle=\langle J_{\emptyset}(W)\rangle,
$$
so (\ref{eq10}) holds.

For every $A\subseteq X$, define
$$
Z_A:=\sum_{B\subseteq A}(-1)^{|B|}\langle J_B(W)\rangle.
$$
By M\"obius inversion, we get
$$
\langle J_B(W)\rangle=\sum_{A\subseteq B}(-1)^{|A|}Z_A.
$$
By induction we may assume that
$$
\big(\sigma^{n-\rk(A)}{\mathcal P}[{\bf X}\mid_A]\big)\star W=Z_A
$$
for all proper subsets $A\subset X$.

Let us assume that $V_X=(0)$. From (\ref{eq10.2})  and Proposition~\ref{prop9} follows that
$Z_X$ is a combination of $\langle S_{\lambda}(W)\rangle$
with $|\lambda|<d$.
Consider
\begin{multline}\label{eq11}
\big(\sigma^n{\mathcal H}[X](\sigma^{-1},-1)\big)\star W-\langle J(W)\rangle=\\
=
\sum_{A\subseteq X}(-1)^{|A|} \big(\sigma^{n-\rk(A)}{\mathcal P}[{\bf X}\mid_A]\big)\star W-\langle J(W)\rangle=\\
=(-1)^{|X|}\big(\sigma^{n-\rk(X)}{\mathcal P}[{\bf X}]\star W-Z_X\big)
+\sum_{A\subseteq X}(-1)^{|A|} Z_A-\langle J(W)\rangle=\\
=(-1)^{|X|}\big(\sigma^{n-\rk(X)}{\mathcal P}[{\bf X}]\star W-Z_X\big).
\end{multline}
In $\big(\sigma^n{\mathcal H}[X](\sigma^{-1},-1)\big)\star W$ and $\langle J(W)\rangle$
only terms $\langle S_{\lambda}(W)\rangle$ appear with $|\lambda|\geq d$.
On the other hand, in $\sigma^{n-\rk(X)}{\mathcal P}[{\bf X}]\star W$ and $Z_X$ only terms $\langle S_{\lambda}(W)\rangle$
appear with $|\lambda|<d$. It follows that the left-hand side and the right-hand side of (\ref{eq11})
are equal to 0.

Suppose that $V_X\neq (0)$. Let $V'$ be a complement of $V_X$ in $V$ of dimension $n-r(X)$.
Define $V'_x=V'\cap V_x$ for all $x\in X$ and $V'_A=V'\cap V_A=\bigcap_{x\in A} V'_x$ for all $A\subseteq X$. 
We have that $V'_X=V'\cap V_X=(0)$ and $V'_A=V'_A\oplus V_X$ for all $A\subseteq X$.
It follows that
$$
\rk(A)=\dim V-\dim V_A=(\dim V'+\dim V_X)-(\dim V'_A+\dim V_X)=\dim V'-\dim V'_A.
$$
Let $J_x'(W)\subseteq K[V'\otimes W^\star]$ be the vanishing ideal of $V_x'\otimes W^\star$ inside $V'\otimes W^\star$. 
Define $J'_A(W)=\prod_{x\in A}J_x'(W)$ and set $J'(W)=J'_X(W)$. 
By the previous case,
$$
\big(\sigma^{\rk(X)}{\mathcal H}[X](\sigma^{-1},-1)\big)\star W=\langle J'(W)\rangle  
$$
and
$$
{\mathcal P}[{\bf X}]\star W=\sum_{A\subseteq X}(-1)^{|A|}\langle J'_A(W)\rangle.
$$
It follows that
$$
J(W)=J'(W)\otimes S(V_X^\star\otimes W)=J'(W)\otimes S(W)^{\otimes (n-\rk(X))}
$$
and
\begin{multline}\label{eq20}
\big(\sigma^{n}{\mathcal H}[X](\sigma^{-1},-1)\big)\star W=
\big(\sigma^{\rk(X)}{\mathcal H}[X](\sigma^{-1},-1)\big)\star W\otimes \langle S(W)^{\otimes (n-\rk(X))}\rangle
=\\
=\langle J'(W)\otimes S(W)^{\otimes (n-\rk(X))}\rangle=\langle J(W)\rangle.
\end{multline}
Similarly, from 
$${\mathcal P}[{\bf X}]\star W=\sum_{A\subseteq X}(-1)^{|A|}\langle J_A'(W)\rangle$$
follows
$$(\sigma^{n-\rk(X)}{\mathcal P}[{\bf X}])\star W=
\sum_{A\subseteq X}(-1)^{|A|}\langle J_A'(W)\otimes S(W)^{\otimes (n-\rk(X))}\rangle=
\sum_{A\subseteq X}(-1)^{|A|}\langle J_A(W) \rangle. 
$$

\end{proof}
\begin{corollary}\label{cor11}
Suppose that $V_X=(0)$.
If we write 
$$
{\mathcal P}[{\bf X}]=u_0-u_1+u_2-\cdots+(-1)^{d-1}u_{d-1}
$$
where $u_i$ is a homogeneous symmetric polynomial of degree $i$ for all $i$,
then
$$
u_i\star W=\langle H_i(W)\rangle.
$$
\end{corollary}

\begin{proposition}\label{prop12}
We can write 
$$
{\mathcal H}[{\bf X}](\sigma^{-1},-1)=w_d-w_{d+1}+w_{d+2}-w_{d+3}+\cdots
$$
where $d=|X|$ and $w_i$ is a homogeneous symmetric polynomial of degree $i$.
We have
$$
w_{d+i}\star W=\langle \Tor_i(J(W),K)\rangle.
$$
\end{proposition}
\begin{proof}
Since $J(W)$ is $d$-regular and generated in degree $d$, it has a linear minimal free resolution. 
We can choose this resolution to be $\GL(W)$-equivariant.
Define
$$
E_i(W):=\Tor_i(J(W),K).
$$
The minimal resolution has the form
$$
0\to E_\ell(W)\otimes R(W)\to  \cdots \to E_1(W)\otimes R(W)\to E_0(W)\otimes R(W)\to J(W)\to 0.
$$
where $\ell=\pd(J(W))$ is the projective dimension of $J(W)$. 
We have
$$
\big(\sigma^n{\mathcal H}[{\bf X}](\sigma^{-1},-1)\big)\star W=\langle J(W)\rangle=\sum_{i=0}^\ell (-1)^i\langle E_i(W)\otimes R(W)\rangle
$$
so
$$
{\mathcal H}[{\bf X}](\sigma^{-1},-1)\star W=\big(\sum_{i=0}^\infty (-1)^iw_{d+i}\big)\star W=\sum_{i=0}^\ell (-1)^i \langle E_i(W)\rangle
$$
\end{proof}
\begin{example}
Let $V=\C$ and let $V_1=V_2=\cdots=V_d=\{0\}$. The rank function is the same as in Example~\ref{exmedges}.
$$
{\mathcal H}[{\bf X}](q,t)=1+q\sum_{i=1}^d{d\choose i}t^i\left(\sum_{j=0}^{i-1}(-1)^j{i-1\choose j}s_j\right).
$$
The ideal $J(W)={\mathfrak m}(W)^d$ where ${\mathfrak m}(W)$ is the maximal homogeneous ideal
in $K[V\otimes W^\star]\cong K[W^\star]\cong S(W)$.

For $d=1$ we have
$$
{\mathcal H}[{\bf X}](q,t)=1+qt,
$$
It follows that
$$
{\mathcal H}[{\bf X}](\sigma^{-1},-1)=1-\sigma^{-1}=s_1-s_{1,1}+s_{1,1,1}-\cdots
$$
This shows that the $i$-th free module in the free resolution
is $S(W)\otimes S_{1,1,\dots,1}W\cong S(W)\otimes \bigwedge^i(W)$.
So the minimal resolution is
$$
\cdots \to S(W)\otimes S_{1,1}(W)\to S(W)\otimes W\to {\mathfrak m}(W)\to 0,
$$
which is of course the Koszul resolution of the maximal ideal ${\mathfrak m}(W)$.
For $d=2$, we get
$$
{\mathcal H}[{\bf X}](q,t)=1+2qt+qt^2(1-s_1)
$$
and
$$
{\mathcal H}[{\bf X}](\sigma^{-1},-1)=1-\sigma^{-1}(1+s_1)=s_2-s_{2,1}+s_{2,1,1}-\cdots
$$
So this means the the equivariant minimal free resolution of ${\mathfrak m}(W)^2$ looks like
$$
\cdots \to S(W)\otimes S_{2,1,1}(W) 
\to S(W)\otimes S_{2,1}(W)\to S(W)\otimes S_2(W)\to {\mathfrak m}(W)^2\to 0.
$$
\end{example}\subsection{Nonnegativity results for the coefficients of ${\mathcal P}[{\bf X}]$ and ${\mathcal H}[{\bf X}](q,t)$}

\begin{corollary}\label{corNonneg}
Suppose that ${\bf X}=(X,\rk)$ is realizable over a field $K$ of characteristic 0.
\begin{enumerate}
\item
 \begin{equation}\label{eq21}
\sigma^{\rk(X)}{\mathcal H}[{\bf X}](\sigma^{-1},-1)=\sum_\lambda a_{\lambda}s_{\lambda}
\end{equation}
where $\lambda$ runs over all partitions with $|\lambda|\geq d$ and $a_{\lambda}\geq 0$ for all $\lambda$;
\item
$$
{\mathcal P}[{\bf X}]=\sum_{\lambda}(-1)^{|\lambda|}b_{\lambda} s_{\lambda}
$$
where $\lambda$ runs over all partitions with $|\lambda|<d$ and $b_{\lambda}\geq 0$ for all $\lambda$;

\item
$$
{\mathcal H}[{\bf X}](\sigma^{-1},-1)=\sum_{\lambda}(-1)^{|\lambda|}c_{\lambda}s_{\lambda}
$$
where $\lambda$ runs over all partitions $\lambda$ with $|\lambda|\geq d$ with
more than $|\lambda|/\rk(X)$ parts, and $c_{\lambda}\geq 0$ for all $\lambda$.
\end{enumerate}
\end{corollary}
\begin{proof}
Assume, as before, that $V$ together with $V_x$, $x\in X$ form a realization of ${\bf X}$.
We may also assume that $V_X=(0)$.

(1) From Remark~\ref{rem5} follows that no $s_\lambda$ with $|\lambda|<d$ appears in the left-hand side of (\ref{eq21}).
If we choose $\dim W\geq |\lambda|$ then $S_{\lambda}(W)\neq 0$ and $\langle S_{\lambda}(W)\rangle$ appears
with a nonnegative coefficient on the right-hand side of (\ref{eq20}). Therefore, the coefficient
of $s_{\lambda}$ in $\sigma^{\rk(X)}{\mathcal H}[{\bf X}](\sigma^{-1},-1)$ is nonnegative.

(2) This follows from Corollary~\ref{cor11}.

(3) The nonnegativity of $c_\lambda$ follows from Proposition~\ref{prop12}.
If $\ell=\pd(J(W))$ is the projective dimension of $J(W)$, then we have
$$\ell=\pd(J(W))=\pd(R(W)/J(W))-1<\dim V\dim W=\rk(X)\dim W
$$
Suppose $\lambda=(\lambda_1,\dots,\lambda_k)$ and the coefficient of $s_\lambda$ in ${\mathcal H}[{\bf X}](\sigma^{-1},-1)$ is nonzero.
If $W$ is $k$-dimensional, then $S_{\lambda}(W)\neq 0$, so $E^{|\lambda|}(W)\neq 0$ and $|\lambda|\leq \ell<\rk(X)k$.
\end{proof}
\begin{conjecture}\label{conNonneg}
Corollary~\ref{corNonneg} is true, even if ${\bf X}=(X,\rk)$ is a polymatroid that is not realizable.
\end{conjecture}
\subsection{The Rees ring and the invariant $\widetilde{H}[{\bf X}](q,t,y)$}
Instead of looking at the $\GL(W)$-Hilbert series of $J(W)$, one could also consider the $\GL(W)$-Hilbert series of the
Rees ring
$$
R(W)[yJ(W)]=R(W)\oplus yJ(W)\oplus y^2J(W)^2\oplus \cdots
$$
where $y$ is an indeterminate. 
This Hilbert series is
$$
\sigma^n\sum_{i=0}^\infty {\mathcal H}[{\bf X}^i](\sigma^{-1},-1)y^i
$$
where
$$
{\bf X}^i=\underbrace{{\bf X}\oplus {\bf X}\oplus\cdots\oplus {\bf X}}_i.
$$
It is therefore natural to define the invariant
$$
\widetilde{{\mathcal H}}[{\bf X}](q,t,y):=\sum_{i=0}^\infty {\mathcal H}[{\bf X}^i](q,t)y^i.
$$
Another interesting ring is the subalgebra $T(W)$ of $R(W)$ generated by
$$
(W\otimes Z_1)(W\otimes Z_2)\cdots (W\otimes Z_d)
$$
The degree $kd$ part in $T(W)$ (or degree $k$ after rescaling) is equal to the degree $(kd,d)$ part in $R(W)$. If we take
$$
\sigma^n\widetilde{\mathcal H}[{\bf X}](\sigma^{-1},-1,z^{-1}),
$$
replace $s_{\lambda}$ by $z^{|\lambda|d}s_{\lambda}$ for all $\lambda$ and then
set $z=0$, then we obtain the Hilbert series of $T(W)$.

It was proven in \cite{Conca} that the algebra $T(W)$ is Koszul when $Z_1,Z_2,\dots,Z_d$ are transversal. If Conjecture~4.2 in that paper
is true, then  $T(W)$ is Koszul
for arbitrary subspaces $Z_1,\dots,Z_d$.
Such a Koszul duality would lead to new interesting
interpretations of the coefficients of $\widetilde{{\mathcal H}}$.

\section{The polarized Schur functor}\label{sec:5}
\subsection{The space $S_{\lambda}(Z_1,\dots,Z_d)$}
Assume again that ${\bf X}=(X,\rk)$ is a polymatroid,
$K$ is a field of characteristic $0$, and that we have a realization 
given by a vector space $V$ and subspaces $V_x$, $x\in X$.
Define $Z=V^\star$, and for every $x\in X$, let $Z_x=V_{x}^{\perp}$ be the set of all linear functionals on $V$ vanishing on $V_x$.
Also, for any $A\subseteq X$, let
$$
Z_A=V_A^{\perp}=\sum_{x\in A} Z_x.
$$
We have
$$
\rk(A)=\dim V-\dim V_A=\dim Z_A
$$
for all $A\subseteq X$.

Let $\Sigma_d$ be the symmetric group on $d$ letters. Its irreducible representations are $T_{\lambda}$
where $\lambda$ runs over all partitions of $d$.

Schur-Weyl duality gives a decomposition
$$
Z^{\otimes d}:=\underbrace{Z\otimes Z\otimes \cdots \otimes Z}_d\cong \bigoplus_{\lambda} S_{\lambda}Z\otimes T_{\lambda}
$$
as a representation of $GL(Z)\times \Sigma_d$.
Let
$$
\pi_{\lambda}:Z^{\otimes d}\to S_{\lambda} Z\otimes T_{\lambda}
$$
be the $GL(Z)\times \Sigma_d$-equivariant projection. 
There is a unique $\GL(Z)\times \Sigma_d$-equivariant
linear map
$$
\theta_{\lambda}:Z^{\otimes d}\otimes T_{\lambda}^\star\to S_{\lambda}(Z)
$$
such that
$$
\theta_{\lambda}(z\otimes \varphi)=(\id\otimes \varphi)\pi_{\lambda}(z)
$$
for every $z\in Z^{\otimes n}$ and $\varphi\in T_{\lambda}^\star$.
Note that $T_{\lambda}^\star\cong T_{\lambda}$ as representations of $\Sigma_d$.
\begin{definition}
We define
$$
S_{\lambda}(Z_1,Z_2,\dots,Z_d)=\theta_{\lambda}(Z_1\otimes Z_2\otimes \cdots \otimes Z_d\otimes T_{\lambda})
$$
\end{definition}
\begin{remark}
For a permutation $\tau\in \Sigma_d$ we have
\begin{multline}
S_{\lambda}(Z_1,\dots, Z_d)=
\theta_{\lambda}(Z_1\otimes \cdots \otimes Z_d)\otimes T_{\lambda})=
\theta_{\lambda}(\tau^{-1}(Z_1\otimes \cdots \otimes Z_d\otimes T_{\lambda}))=\\
\theta_{\lambda}(\tau^{-1}(Z_1\otimes \cdots \otimes Z_d)\otimes T_{\lambda})=
\theta_{\lambda}(Z_{\tau(1)}\otimes \cdots \otimes Z_{\tau(d)})\otimes T_{\lambda})=
S_{\lambda}(Z_{\sigma(1)},\dots,Z_{\sigma(d)}).
\end{multline}
In other words, $S_{\lambda}(Z_1,\dots,Z_d)$ does not depend on the order of $Z_1,\dots,Z_d$.
\end{remark}
Note that
$$
S_{\lambda}(\underbrace{Z,Z,\dots,Z}_d)=S_{\lambda}(Z).
$$
\subsection{The connection between $S_{\lambda}(Z_1,\dots,Z_d)$ and ${\mathcal H}[{\bf X}](q,t)$}
\begin{proposition}
Let us write
$$
\sigma^n{\mathcal H}[{\bf X}](\sigma^{-1},-1)=\sum_{\lambda}a_{\lambda}s_{\lambda}
$$
where $\lambda$ runs over all partitions with $|\lambda|\geq d$.
Then we have
$$
a_{\lambda}=\dim S_\lambda(Z_1,Z_2,\dots,Z_d,\underbrace{Z,\dots,Z}_{|\lambda|-d}).
$$
\end{proposition}
\begin{proof}
Let $r=|\lambda|$ and  ${\mathfrak m}(W)$ be the maximal homogeneous ideal of $R(W)$.
The degree $r$ part of $J(W)$ is 
$$
J_1(W)J_2(W)\cdots J_d(W){\mathfrak m}(W)^{r-d}.
$$
Set $U=W\otimes V^\star=W\otimes Z$ and $U_i=W\otimes Z_i$.
Then Cauchy's formula tells us that
$$R(W)=S(W\otimes Z)=\bigoplus_{\lambda} S_{\lambda}W\otimes S_{\lambda}Z.
$$
The degree $r$ part of $J(W)$ is
$$
U_1\cdot U_2\cdots U_d\cdot U^{r-d}\subset S_r(U)=\bigoplus_{|\lambda|=r} S_{\lambda}W\otimes S_{\lambda}Z.
$$
So if 
$$
\pi_r^U:\underbrace{U\otimes U\otimes \cdots \otimes U}_r\to S_r(U)
$$
is the canonical projection, then the degree $r$ part of $J(W)$
is 
$$
\pi_r(U_1,U_2,\dots,U_d,\underbrace{U,\dots,U}_{r-d}).
$$
Let $\gamma_\lambda:S_r(U)\to S_{\lambda} W\otimes S_\lambda Z$ be the projection.
The isotypic component of $J(W)$ for the representation $S_\lambda(W)$ is
$$
\gamma_\lambda(\pi_r(U_1\otimes \cdots \otimes U_{d}\otimes U^{r-d})).
$$
We have
$$
U^{\otimes r}=(Z\otimes W)^{\otimes r}=Z^{\otimes r}\otimes W^{\otimes r}=
\bigoplus_{\lambda} S_{\lambda}(W)\otimes T_\lambda\otimes Z^{\otimes r}\cong 
\bigoplus_{\lambda} S_{\lambda}(W)\otimes Z^{\otimes r}\otimes T_\lambda.
$$
If we first project $U^{\otimes r}$ onto $S_{\lambda}(W)\otimes Z^{\otimes r}\otimes T_{\lambda}$
and then we apply
$$
\id\otimes\pi_{\lambda}^Z:S_{\lambda}(W)\otimes Z^{\otimes r}\otimes T_{\lambda}\to S_{\lambda}W\otimes S_\lambda Z
$$
then we get a nonzero $\GL(V)\times \GL(Z)\times \Sigma_r$ equivariant linear map
$$
U^{\otimes r}\to S_{\lambda}W\otimes S_{\lambda}Z
$$
This map must be, up to a non-zero scalar, equal to the composition $\gamma_\lambda\circ \pi_r$. 
It follows that
\begin{multline*}
\gamma_\lambda(\pi_r(U_1\otimes \cdots \otimes U_{d}\otimes U^{r-d}))=
\id\otimes\pi_{\lambda}(S_\lambda(W)\otimes Z_1\otimes \cdots \otimes Z_d\otimes Z^{r-d}\otimes T_\lambda)=\\
=
S_{\lambda}W\otimes S_\lambda(Z_1,\dots,Z_d,\underbrace{Z,\dots,Z}_{r-d}).
\end{multline*}
So, as $\GL(W)$-modules, we have an isomorphism
$$
J(W)\cong \bigoplus_{\lambda}S_{\lambda}(Z_1,Z_2,\dots,Z_d,\underbrace{Z,\dots,Z}_{|\lambda|-d})\otimes S_{\lambda}(W).
$$
Since $a_\lambda$ is the multiplicity of $S_\lambda W$ in $J(W)$, we get
$$
a_{\lambda}=\dim  S_\lambda(Z_1,\dots,Z_d,\underbrace{Z,\dots,Z}_{r-d}).
$$
\end{proof}

For $A\subseteq X$, let us define
$$
S_{\lambda,A}:=S_\lambda(V_{x_1},\dots,V_{x_k},\underbrace{V,\dots,V}_{|\lambda|-k})
$$
where $k=|A|$ and $A=\{x_1,\dots,x_k\}$. If $|\lambda|<k$, then we define $S_{\lambda,A}=0$.
Define
$$
C_{\lambda,k}=\bigoplus_{|A|=k} S_{\lambda,A}.
$$
Then we get
$$
C_{k}=\bigoplus_{\lambda}C_{\lambda,k}\otimes S_{\lambda}(W).
$$
The maps in the complex (\ref{eq:ScomplexW}) are $\GL(W)$-equivariant,
and by taking the isotypic component for $S_{\lambda(W)}$
we get a  complex
$$
0\to C_{\lambda,\ell}\otimes S_{\lambda}(W)\to \cdots \to C_{\lambda,1}\otimes S_{\lambda}(W)\to C_{\lambda,0}\otimes S_{\lambda}(W)\to 0
$$
where $\ell=\min\{d,|\lambda|\}$.
Since all maps in this complex are $\GL(W)$-equivariant, the complex is obtaind from a complex
\begin{equation}\label{eq:complex_lambda}
0\to C_{\lambda,\ell}\to \cdots \to C_{\lambda,1}\to C_{\lambda,0}\to 0
\end{equation}
by tensoring it by $S_{\lambda}(W)$.
The map $\partial_k:C_{\lambda,k}\to C_{\lambda,k-1}$ can
be written as $\partial_k=\sum_{A,B}\partial_{k}^{A,B}$, where
$$
\partial_k^{A,B}:S_{\lambda,A}\to S_{\lambda,B}
$$
Suppose that $A=\{i_1,i_2,\dots,i_k\}$ with $i_1<i_2<\cdots < i_k$,
then we have
$$
\partial_k^{A,B}:=
\left\{
\begin{array}{ll}
0 & \mbox{if $B\not\subseteq A$;}\\
(-1)^r\id & \mbox{if
$B=\{i_1,\dots,i_{r-1},i_{r+1},\dots,i_k\}$.}\end{array}\right.
$$

Let $H_{\lambda,i}$ be the $i$-th homology group of (\ref{eq:complex_lambda}).
From
$$
H_i(W)=\bigoplus_{\lambda}H_{\lambda,i}\otimes S_{\lambda}(W).
$$
and Corollary~\ref{cor11} now follows the following statement.
\begin{corollary}
Suppose that $V_X=0$, which means that
 $Z_X=Z$.
Write 
$${\mathcal P}[{\bf X}]=\sum_{\lambda}(-1)^{|\lambda|}b_{\lambda}s_{\lambda}.
$$
Then we have
$$
\dim H_{\lambda,i}=\left\{
\begin{array}{ll}
0 & \mbox{if $|\lambda|\neq i$;}\\
b_{\lambda} & \mbox{if $|\lambda|=i$.}
\end{array}\right.
$$
\end{corollary}
The dimension of
$$
S_{\lambda}Z=S_{\lambda}(\underbrace{Z,Z,\dots,Z}_{d})
$$
(where $d=|\lambda|$)
is exactly the number of Young Tableau of shape $\lambda$ and entries
in the set $\{1,2,\dots,n\}$. In fact, given a basis of $Z$, an explicit
basis of $S_{\lambda}Z$ can be given in terms of these Young tableaux
(see~\cite[\S8.1, Theorem 1]{Fulton}).
\begin{problem}
Give an combinatorial interpretation of
$$
\dim S_\lambda(Z_1,Z_2,\dots,Z_d),
$$
perhaps in terms of certain fillings of Young diagrams. 
Moreover, can one give an explicit basis of $S_{\lambda}(Z_1,\dots,Z_d)$?
\end{problem}
Such a combinatorial
setup might still have a meaning for non-realizable polymatroids.
An explicit bases of $S_d(Z_1,\dots,Z_d)$ was given in \cite[Corollary 5.10]{Conca}
in case the subspaces $Z_1,\dots,Z_d$ of $Z$ are generic.

Also, one can ask the same questions for $H_{\lambda}:=H_{\lambda,|\lambda|}$.
Such results might prove Conjecture~\ref{conNonneg}.

\section{Quasi-symmetric functions associated to polymatroids}
\subsection{The Hopf algebras $\Mat$ and $\PolyMat$}
Although most of the Hopf algebras in this section can be defined
over the integers $\Z$, we will choose to define them over $\Q$ for simplicity. In \cite{Schmitt} the matroid Hopf algebra $\Mat$
was introduced (see also \cite{CS,CS2,CS3}). This construction easily generalizes to polymatroids.

Let us first introduce the Hopf algebra of polymatroids, $\PolyMat$. 
For a polymatroid ${\bf X}=(X,\rk)$, we denote its isomorphism class by $[{\bf X}]$.
As a $\Q$-vector space, $\PolyMat$ has a basis consisting of all isomorphism classes
of polymatroids.
We define a product by
$$
[{\bf X}]\cdot [{\bf Y}]:=[{\bf X}\oplus{\bf Y}].
$$
Also, a coproduct $\Delta:\PolyMat\to \PolyMat\otimes_{\Q} \PolyMat$ is defined by
$$
\Delta[{\bf X}]=\sum_{A\subseteq X} [{\bf X}\mid_A]\otimes [{\bf X}/A].
$$
This coproduct is coassociative, but in general not cocommutative.
The unit is $[\emptyset]$ where $\emptyset$ denotes the empty polymatroid.
A counit $\epsilon:\PolyMat\to \Q$ is given by
$$\epsilon([{\bf X}])=\left\{
\begin{array}{ll}
1 & \mbox{if ${\bf X}=\emptyset$}\\
0 & \mbox{otherwise.}\end{array}\right.
$$
The bialgebra $\PolyMat$ has a grading such that $[{\bf X}]$ has degree $|X|$ for every polymatroid ${\bf X}=(X,\rk)$.
This makes $\PolyMat$ into a connected graded bialgebra. It was shown in \cite{MiMo} that one can
define an antipode such that $\PolyMat$ becomes a Hopf algebra.

Let $\Mat$ be the subspace spanned by all $[{\bf X}]$ where ${\bf X}$
is  a matroid. Then $\Mat$ is sub-Hopf algebra of $\PolyMat$.

\subsection{The Hopf algebra $\NSym$}
Let $\NSym\Q\langle p_1,p_2,p_3,\dots\rangle$ be the ring
of noncommutitive polynomials in the indeterminates $p_1,p_2,p_3,\dots$.
We define a Hopf algebra structure on $\NSym$ as follows.
The  comultiplication $\Delta:\NSym\to\NSym\otimes \NSym$ by
$$
\Delta(p_i)=p_i\otimes 1+1\otimes p_i
$$
for all $i$. The counit $\epsilon: \NSym\to \Q$ is defined by
$$
\epsilon(p_i)=0
$$
for all $i$. The antipode is defined
by
$$
p_i\mapsto -p_i
$$
for all $i$. A basis of $\NSym$ is given by all noncommutative monomials in $p_1,p_2,\dots$.
It is also convenient to have a different basis. We define $h_1,h_2,\dots$ by
the following equality of generating functions in $\NSym[[t]]$.
Define
$$
H(t)=h_1t+h_2t^2+h_3t^3+\cdots
$$
and 
$$
P(t)=p_1t+p_2t^2+p_3t^3+\cdots.
$$
Then $h_1,h_2,h_3,\dots$ are defined by
$$
1+H(t)=\exp(P(t)).
$$
Here $\exp(t)$ denotes the  power series of the exponential function
$$
\exp(t)=1+t+\frac{t^2}{2!}+\frac{t^3}{3!}+\cdots.
$$
So we have 
\begin{equation}\label{eq:hk}
h_k=\sum_{r=1}^k\frac{1}{r!}\big(\sum_{i_1,\dots,i_r\atop
i_1+\cdots+i_r=k}p_{i_1}p_{i_2}\cdots p_{i_r}\big).
\end{equation}
If $\alpha=(i_1,\dots,i_r)$ is a sequence of positive integers, then we will write 
$p_{\alpha}$ instead of $p_{i_1}p_{i_2}\cdots p_{i_r}$
and $h_{\alpha}$ instead of $h_{i_1}h_{i_2}\cdots h_{i_r}$.
. The length of $\alpha$ is
$\ell(\alpha):=r$, and we define $|\alpha|=i_1+i_2+\cdots+i_r$.
We can rewrite (\ref{eq:hk}) as
\begin{equation}\label{eq:hk2}
h_k=\sum_{\alpha\atop |\alpha|=k}\frac{p_{\alpha}}{\ell(\alpha)!}
\end{equation}
Inverting gives
$$
P(t)=\log(1+H(t))
$$
where
$$
\log(1+t)=t-\frac{t^2}{2}+\frac{t^3}{3}-\cdots,
$$
so
$$
p_k=\sum_{r=1}^k\frac{(-1)^{r-1}}{r}\sum_{i_1,\dots,i_r\atop
i_1+\cdots+i_r=k}h_{i_1}h_{i_2}
\cdots h_{i_r}.
$$
Again, we can rewrite this as
\begin{equation}\label{eq:pk2}
p_{k}=\sum_{\alpha}\frac{(-1)^{\ell(\alpha)-1}h_{\alpha}}{\ell(\alpha)}.
\end{equation}

From
$$
\Delta(P(t))=P(t)\otimes 1+1\otimes P(t)
$$
follows that
\begin{multline*}
\Delta(1+H(t))=\Delta(\exp(P(t))=\Delta(\exp(P(t)\otimes 1+1\otimes P(t)))=\\
=
\exp(P(t)\otimes 1)\exp(1\otimes P(t))=
((1+H(t))\otimes 1)\cdot (1\otimes (1+H(t)))=\\
=(1+H(t))\otimes (1+H(t))
\end{multline*}
inside the ring
$$
\NSym\otimes \NSym[[t]]=\NSym[[t]]\otimes_{\Q[[t]]}\NSym[[t]].
$$
If we use the convention $h_0=1$, then we have
$$
\Delta(h_k)=\sum_{i=0}^k h_i\otimes h_{k-i}.
$$
The Hopf algebra $\NSym$ is not commutative, but it is cocommutative.

\subsection{The Hopf algebra $\QSym$}
Let $\QSym$ be the Hopf algebra of quasi-symmetric functions.
For a sequence $\alpha=(\alpha_1,\dots,\alpha_r)$
of positive integers we define
an element $M_{\alpha} \in \Q[x_1,x_2,\dots]$ by
$$
M_{\alpha}:=\sum_{0<i_1<i_2<\cdots<i_r}x_1^{\alpha_1}x_2^{\alpha_2}\cdots x_r^{\alpha_r}.
$$
 The ring $\QSym$ is the subring of $\Q[x_1,x_2,x_3,\dots]$ spanned by all $M_{\alpha}$.
The $\Q$-vector space  $\QSym$ is closed under multiplication.
We will view $\QSym$ as the graded dual vector space of $\NSym$
where the $\{M_{\alpha}\}$ form a dual basis of the $\{h_{\alpha}\}$.
As such, $\QSym$ is a Hopf algebra in a natural way. Also, let $\{P_{\alpha}\}$
be a dual basis of $\{p_{\alpha}\}$. 
We have that
$$
P_{\alpha}P_{\beta}=\sum_{\gamma} P_{\gamma}
$$
Where $\gamma$ runs over all 
$$
{\ell(\alpha)+\ell(\beta)\choose \ell(\alpha)}
$$
shuffles of $\alpha$ and $\beta$.
If $\alpha=(\alpha_1,\dots,\alpha_r)$, then
$$
\Delta(P_{\alpha})=\sum_{\beta,\gamma;\beta\gamma=\alpha} P_{\beta}\otimes P_{\gamma}.
$$
The antipode on $\QSym$ is given by
$$
P_{\alpha}\mapsto (-1)^{\ell(\alpha)}P_{\alpha}.
$$
From (\ref{eq:hk2}) follows that
\begin{equation}\label{eq:hk3}
h_{\alpha}=h_{i_1}\cdots h_{i_r}=\sum_{\beta_1,\dots,\beta_r\atop
|\beta_1|=i_1,\dots,|\beta_r|=i_r} \frac{p_{\beta_1\beta_2\cdots \beta_r}}{\ell(\beta_1)!\cdots \ell(\beta_r)!},
\end{equation}
where $\alpha=(i_1,\dots,i_r)$.
Dualizing (\ref{eq:hk3}) gives
$$
P_{\beta}=\sum_r \sum_{\beta_1\cdots \beta_r\atop \beta=\beta_1\cdots\beta_r}\frac{M_{
|\beta_1|,\dots,|\beta_r|}}
{\ell(\beta_1)!\ell(\beta_2)!\cdots \ell(\beta_r)!}.
$$
From (\ref{eq:pk2}) follows that
\begin{equation}\label{eq:pk3}
p_{\alpha}=p_{i_1}\cdots p_{i_r}=\sum_{\beta_1,\dots,\beta_r\atop
|\beta_1|=i_1,\dots,|\beta_r|=i_r} \frac{(-1)^{\ell(\beta_1)+\cdots+\ell(\beta_r)-r}h_{\beta_1\beta_2\cdots \beta_r}}{\ell(\beta_1)\cdots \ell(\beta_r)}.
\end{equation}
Dualizing (\ref{eq:pk3}) yields
and
\begin{equation}\label{eq:pk4}
M_{\beta}=\sum_r \sum_{\beta_1\cdots \beta_r\atop \beta=\beta_1\cdots\beta_r}(-1)^{\ell(\beta)-r}
\frac{P_{|\beta_1|,\dots,|\beta_r|}}{\ell(\beta_1)\cdots \ell(\beta_r)}.
\end{equation}

\subsection{Combinatorial Hopf algebras and the invariant ${\mathcal F}[{\bf X}]$}\label{sec:7.4}
Billera, Jia and Reiner defined a homomorphism of Hopf algebras
$$
{\mathcal F}:\Mat\to \QSym.
$$
One way to define this map is using a universal property of $\QSym$.

A combinatorial Hopf algebra (over $\Q$) is a pair $({\mathcal H},\zeta)$
where ${\mathcal H}=\bigoplus_{d\geq 0} {\mathcal H}_d$
is a graded Hopf algebra with ${\mathcal H}_0=\Q$ and ${\mathcal H}_d$ is
finite dimensional for all $d$, and $\zeta:{\mathcal H}\to \Q$ is a character
(i.e., a algebra homorphism).
A morphism $\varphi:({\mathcal H}',\zeta')\to ({\mathcal H},\zeta)$
is a Hopf-algebra morphism $\varphi:{\mathcal H}'\to {\mathcal H}$
such that $\zeta\circ \varphi=\zeta'$.

Aguiar, Bergeron and Sottile proved that in there exists a terminal
object in the category of combinatorial Hopf algebras over $\Q$,
namely $(\QSym,\zeta)$ where
$\zeta=\zeta_{\QSym}$ is defined by
$$
\zeta(M_{\alpha})=\left\{
\begin{array}{ll}
1 & \mbox{if $\ell(\alpha)\leq 1$};\\
0 & \mbox{otherwise.}\end{array}\right.
$$
We can define a character $\zeta=\zeta_{\Mat}$ on $\Mat$ by
$$
\zeta([{\bf X}])=\left\{
\begin{array}{ll}
1 & \mbox{if ${\bf X}$ completely splits in to loop and coloop matroids;}\\
0 & \mbox{otherwise.}\end{array}\right.
$$
Since $(\QSym,\zeta_{\QSym})$ is terminal, there is a unique
homomorphism 
$${\mathcal F}:(\Mat,\zeta_M)\to (\QSym,\zeta_{\QSym})$$
of combinatorial Hopf algebras.

Although ${\mathcal F}$ is a powerful invariant for matroids, it cannot distinguish
between a loop and an isthmus.
\subsection{The new quasi-symmetric function invariant ${\mathcal G}[{\bf X}]$} 
It sometimes is convention to shift the indices by $1$, so for
a vector $a=(a_1,a_2,\dots,a_d)$ of nonnegative integers,
we define
$$
U_{(a_1,a_2,\dots,a_d)}:=P_{a_1+1,a_2+1,\dots,a_d+1}.
$$
\begin{definition}
We define a $\Q$-linear map
$$
{\mathcal G}:\PolyMat\to \QSym
$$
defined by
$$
{\mathcal G}[{\bf X}]=\sum_{\underline{X}}U_{r(\underline{X})},
$$
where $\underline{X}$ runs over all maximal chains
$$
\underline{X}:\emptyset=X_0\subset X_1\subset \cdots \subset X_d=X
$$
and 
$$
r(\underline{X}):=(\rk(X_1)-\rk(X_0),\rk(X_2)-\rk(X_1),\dots,\rk(X_d)-\rk(X_{d-1})).
$$
\end{definition}
We call $r(\underline{X})$ the rank sequence for $\underline{X}$. The multiset of all $r(\underline{X})$
where $\underline{X}$ runs over all maximal chains in $X$, we will call {\em the rank sequences for ${\bf X}$}.
If ${\bf X}=(X,\rk)$ then there are exactly $|X|!$ rank sequences.

\begin{lemma}
The linear map ${\mathcal G}$ is a homomorphism of Hopf algebras.
\end{lemma}
\begin{proof}
If ${\bf X}$ has a rank sequence $\gamma=r(\underline{X})$ and $\gamma=\alpha\beta$, then
$\alpha$ is a rank sequence for ${\bf X}\mid_A$ and $\beta$ is a rank sequence for ${\bf X}/A$,
where $A=X_i$ and $i=\ell(\alpha)$ is the length of $\alpha$. So we have
\begin{multline}
{\mathcal G}\otimes {\mathcal G}\circ \Delta([{\bf X}])=\sum_{A\subseteq X}
{\mathcal G}[{\bf X}\mid_A]\otimes {\mathcal G}[{\bf X}/A]=\\=
\sum_{A\subseteq X} \sum_{\alpha}\sum_{\beta}U_{\alpha}\otimes U_{\beta}=
\Delta(\sum_{\gamma}U_{\gamma})=\Delta({\mathcal G}[{\bf X}]),
\end{multline}
where $\alpha$ runs over all rank sequences for ${\bf X}\mid_A$, $\beta$ runs over all rank sequences
of ${\bf X}/A$ and $\gamma$ runs over
all rank sequences for ${\bf X}$.

To see that ${\mathcal G}$ commutes with the product, note that the rank sequences for ${\bf X}\oplus {\bf Y}$
are exactly all shuffles of rank sequences for ${\bf X}$ and ${\bf Y}$.

It easy to verify that ${\mathcal G}$ is compatible with the unit and counit.
\end{proof}
For a vector $\alpha=(\alpha_1,\alpha_2,\dots,\alpha_d)$, define
$$
\alpha^{\vee}=(1-\alpha_d,1-\alpha_{d-1},\dots,1-\alpha_1).
$$
\begin{lemma}
For a matroid ${\bf X}=(X,\rk)$ we have
$$
{\mathcal G}[{\bf X}^{\vee}]=\sum_{\underline{X}}U_{r(\underline{X})^\vee}
$$
\end{lemma}
\begin{proof}
For a maximal chain $\underline{X}$, define a chain $\underline{X}^{\vee}$ by $X^\vee_i:=X\setminus X_{d-i}$.
Note that
$$
\rk^{\vee}(X_i)=|X|-\rk(X)+\rk(X_{d-i})
$$
and
$$
\rk^{\vee}(X^{\vee}_i)-\rk^{\vee}(X^{\vee}_{i-1})=1-(\rk(X_{d-i+1})-\rk(X_{d-i})).
$$
If  $\alpha$ runs over all rank sequence for ${\bf X}$, then $\alpha^{\vee}$ runs over all rank sequences for ${\bf X}^\vee$.
\end{proof}
\subsection{${\mathcal G}$ specializes to ${\mathcal F}$}
Let us define another character $\gamma:\QSym\to \Q$ by
$$
\gamma(P_{\alpha})=0
$$
if $\alpha$ is not weakly increasing.
Otherwise, write $\alpha=(\alpha_1^{k_1},\alpha_2^{k_2},\cdots,\alpha_s^{k_s})$
with
$$
\alpha_1<\alpha_2<\cdots<\alpha_s,
$$
and define
$$
\gamma(P_{\alpha})=\frac{1}{k_1!k_2!\cdots k_s!}.
$$

Suppose that $\alpha'=(\alpha_1^{l_1},\cdots,\alpha_s^{l_s})$.
Then 
$$
P_{\alpha}P_{\alpha'}={l_1+k_1\choose k_1}{l_2+k_2\choose k_2}\cdots {l_s+k_s\choose k_s}
P_{\delta}+P'
$$
where $\delta=(\alpha_1^{k_1+l_1},\dots,\alpha_s^{k_s+l_s})$ and
$P'$ is a linear combination of $P_{\delta}$'s where $\delta$ is not weakly increasing.
The binomials appear from the fact there there are ${l_i+k_i\choose k_i}$ ways
to shuffle $\alpha_i^{k_i}$ and $\alpha_i^{l_i}$.
If we apply $\gamma$ we get
$$
\gamma(P_{\alpha}P_{\alpha'})=\gamma(P_{\delta})=
\frac{{l_1+k_1\choose k_1}\cdots {l_s+k_s\choose k_s}}{(l_1+k_1)!\cdots (l_s+k_s)!}=
\frac{1}{k_1!\cdots k_s!}\cdot \frac{1}{l_1!\cdots l_s!}=
\gamma(P_{\alpha})\gamma(P_{\alpha'}).
$$
This shows that $\gamma$ is multiplicative.
Since $(\QSym,\zeta)$ is the terminal object for
the combinatorial Hopf algebras, there is a unique morphism of combinatorial Hopf algebras
$$
\theta:(\QSym,\gamma)\to (\QSym,\zeta).
$$
\begin{theorem}
We have
$$
\theta\circ {\mathcal G}\mid_{\Mat}={\mathcal F},
$$
where ${\mathcal G}\mid_{\Mat}$ is the restriction of ${\mathcal G}$ to $\Mat$.
\end{theorem}
\begin{proof}
We claim that 
$$
\zeta=\gamma\circ {\mathcal G}\mid_{\Mat}.
$$
Suppose that ${\bf X}=(X,\rk)$ is a matroid with $d:=|X|$ and $n:=\rk(X)\leq d$.
Then
$\gamma({\mathcal G}[{\bf X}])$
is equal to 
$\frac{N}{n!(d-n)!}$,
 where $N$
counts the number of maximal chains
$$
X_0=\emptyset \subset X_1\subset \cdots \subset X_d=X
$$
with 
\begin{equation}\label{eq:cond1}
0=\rk(X_0)=\cdots =\rk(X_{d-n})=0
\end{equation}
and
\begin{equation}\label{eq:cond2}
\rk(X_{d-n+i})=i
\end{equation}
for $i=1,2,\dots,n$.
Let $Y=X_{d-n}$ and $Z=X\setminus Y$. For a subset $A\subseteq X$, we have
$$
\rk(A)\geq \rk(X)-\rk(X\setminus A)
$$
and
$$
\rk(X\setminus A)\leq \rk(Y\setminus A)+\rk(Z\setminus A)=\rk(Z\setminus A)\leq |Z|-|Z\cap A|= n-|Z\cap A|.
$$
It follows that 
$$
\rk(A)\geq \rk(X)-\rk(X\setminus A)=n-(n-|Z\cap A|)=|Z\cap A|.
$$
We also have 
$$
\rk(A)\leq \rk(A\cup Y)\leq \rk(Y)+|A\cup Y|-|Y|=|A\cap Z|
$$
We conclude that
$$
\rk(A)=|A\cap Z|
$$
for all $A\subseteq X$.
This implies that
\begin{equation}\label{eq:split}
(X,\rk)=\underbrace{{\bf 0}\cdot {\bf 0}\cdots {\bf 0}}_{d-n} \cdot \underbrace{{\bf 1}\cdot {\bf 1}\cdots {\bf 1}}_n.
\end{equation}
where ${\bf 0}$ is the loop matroid, and ${\bf 1}$ is the isthmus matroid.
In particular, if $(X,\rk)$ does not split completely, then $\gamma({\mathcal G}[{\bf X}])=0$.

Suppose that ${\bf X}=(X,\rk)$ splits completely as in (\ref{eq:split}).
Without loss of generality, we may assume that $X=\{1,2,\dots,d\}$,
and $\rk(A)=|A\cap Z|$ where $Y=\{1,2,\dots,d-n\}$ and $Z=X\setminus Y$.

A flag 
$$
X_0=\emptyset\subset X_1\subset \cdots \subset X_d=X
$$
satisfies (\ref{eq:cond1}) and (\ref{eq:cond2}) 
if and only if $X_{d-n}=Y$. There are
$(d-n)!$ flags 
$$\emptyset =X_0\subset \cdots \subset X_{d-n}=Y
$$
and $n!$ flags
$$
Y=X_{d-n}\subset X_{d-n+1}\subset \cdots X_{d}=X.
$$
It follows that $N=n!(d-n)!$, and 
$$\gamma({\mathcal G}[{\bf X}])=\frac{N}{n!(d-n)!}=1.
$$
It follows that $\gamma\circ{\mathcal G}\mid_{\Mat}=\zeta=\zeta([{\bf X}])$.
By the uniqueness, we get
$\theta\circ{\mathcal G}\mid_{\Mat}={\mathcal F}$.
\end{proof}
Note that
$$
{\mathcal G}(\Mat)\subseteq \QSym_2
$$
where $\QSym_2$ is the sub-Hopf algebra of $\QSym$ spanned by all $Q_\alpha$'s
where $\alpha$ is a sequences of $0$'s and $1$'s.
The algebra $\QSym_2$ is the graded dual of the Hopf algebra $\Q\langle p_1,p_2\rangle$.
Now $\theta$ restricts to a homomorphism 
$$\theta_2:\QSym_2\to \QSym.
$$
\begin{proposition}
The homomorphism $\theta_2$ is surjective, and the kernel of $\theta_2$
is the principal ideal generated by $P_{(2)}-P_{(1)}=U_{(1)}-U_{(0)}$.
\end{proposition}
\begin{proof}
The surjectivity follows from the fact that ${\mathcal F}$ is surjective.
We choose the grading on $\QSym_2$ where 
$P_{\alpha}$ has degree $\ell(\alpha)$.
There are $2^d$ basis elements $P_{\alpha}$ of degree $d$. 
So the Hilbert series of the $\QSym_2$ is
$$
1+2t+2^2t^2+\cdots =\frac{1}{1-2t}.
$$
Note that $\QSym_2$ is not finitely generated as a commutative algebra.

On $\QSym$, we choose the grading where $P_{\alpha}$ has degree $|\alpha|$.
There is one basis element of degree $0$, namely $P_{()}$ and for $d>0$
there are $2^{d-1}$ basis elements of degree, because there
are $2^{d-1}$ decompositions of $d$.
So the Hilbert series of $\QSym$ with this grading is
$$
1+t+2t^2+2^2t^3+\cdots 1+\frac{t}{1-2t}=\frac{1-t}{1-2t}.
$$
Therefore, the Hilbert series of the kernel of $\theta_2$ is
$$
\frac{1}{1-2t}-\frac{1-t}{1-2t}=\frac{t}{1-2t}.
$$
The kernel contains the principal ideal $(P_{(2)}-P_{(1)})$.
It is not hard to see that $P_{(2)}-P_{(1)}$ is not a zero divisor, so
the Hilbert series of the principal ideal is $\frac{t}{1-2t}$.
Since this is equal to the Hilbert series of the kernel of $\theta_2$
we must have
$$
\ker\theta_2=(P_{(2)}-P_{(1)}).
$$
\end{proof}
\subsection{${\mathcal G}$ specializes to ${\mathcal H}$}
\begin{theorem}
There exists a homomorphism $\tau:\QSym\to Sym[q,t]$ of commutative
algebras such that $\tau({\mathcal G}[{\bf X}])={\mathcal H}[{\bf X}]$
for every polymatroid ${\bf X}$.
\end{theorem}
\begin{proof}
We will inductively define a symmetric function ${\mathcal P}(\alpha)$ for
any vector $\alpha=(\alpha_1,\dots,\alpha_d)$ of nonnegative integers
as follows. We define ${\mathcal P}()=1$. Then ${\mathcal P}(\alpha_1,\dots,\alpha_d)$ is the unique symmetric function of
degree $<d$ such that
\begin{equation}\label{eq:vanishes}
\sum_{i=0}^d {d\choose i}{\mathcal P}(\alpha_1,\dots,\alpha_i)(-1)^i\sigma^{-\alpha_1-\cdots-\alpha_i}
\end{equation}
vanishes in degree $<d$.
For a vector $\alpha=(\alpha_1,\dots,\alpha_d)$ and $i<d$,
let $\alpha^{[i]}=(\alpha_1,\dots,\alpha_i)$ be the truncated vector.
So (\ref{eq:vanishes}) becomes
 \begin{equation}\label{eq:vanishes2}
\sum_{i=0}^d {d\choose i}{\mathcal P}(\alpha^{[i]})(-1)^i\sigma^{-|\alpha^{[i]}|}.
\end{equation}

Define
\begin{equation}\label{eq:Ptilde}
\widetilde{\mathcal P}[{\bf X}]=\frac{1}{d!}\sum_{\underline{X}}{\mathcal P}(r(\underline{X}))
\end{equation}
for every polymatroid ${\bf X}=(X,\rk)$ such that $d=|X|$.
Here $\underline{X}$ runs over all maximal chains in $X$.

We claim that ${\mathcal P}[{\bf X}]=\widetilde{\mathcal P}[{\bf X}]$.
The claim is clearly true when $|X|=0$ or $|X|=1$.
Note that $\widetilde{\mathcal P}[{\bf X}]$ is a symmetric
polynomial of degree $<d=|X|$.
To prove the claim it suffices to show that
\begin{equation}\label{eq:Ptilde}
\sum_{A\subseteq X}\widetilde{\mathcal P}[{\bf X}\mid_A](-1)^{|A|}\sigma^{-\rk(A)}
\end{equation}
vanishes in degree $<d$.
The symmetric polynomial (\ref{eq:Ptilde}) is equal to
\begin{equation}\label{eq:underlineA}
\sum_{i=0}^d\sum_{A\subseteq X\atop |A|=i}
\frac{1}{i!}\sum_{\underline{A}}{\mathcal P}(r(\underline{A}))(-1)^i
\sigma^{-\rk(A)}
\end{equation}
where $\underline{A}$ runs over all maximal chains in $A$. 
Every such chain $\underline{A}$ can be extended to $(d-i)!$
maximal chains in $X$. Therefore, (\ref{eq:underlineA}) is equal
to
\begin{multline}
\sum_{i=0}^d
\frac{1}{i!(d-i)!}\sum_{\underline{X}}{\mathcal P}(r(\underline{X})^{[i]})(-1)^i
\sigma^{-|r(\underline{X})^{[i]}|}=\\
\frac{1}{d!}\sum_{\underline{X}}\sum_{i=0}^d
{d\choose i}{\mathcal P}(r(\underline{X})^{[i]})(-1)^i
\sigma^{-|r(\underline{X})^{[i]}|}
\end{multline}
which vanishes in degree $<d$.

For a vector $\alpha=(\alpha_1,\dots,\alpha_d)$,
define
$$
\tau(U_{\alpha})=\sum_{i=0}^d \frac{1}{i!(d-i)!}
{\mathcal P}(\alpha^{[i]})q^{|\alpha^{[i]}|}t^i.
$$
If ${\bf X}=(X,\rk)$ is a polymatroid with $|X|=d$, then
we have
\begin{equation}\label{eq:tauG}
\tau({\mathcal G}[{\bf X}])=
\tau(\sum_{\underline{X}}U_{r(\underline{X})})=
\sum_{\underline{X}}\tau(U_{r(\underline{X})})=
\sum_{\underline{X}}
\sum_{i=0}^d\frac{1}{i!(d-i)!}{\mathcal P}(r(\underline{X})^{[i]})q^{|r(\underline{X})^{[i]}|}t^i.
\end{equation}
For every subset $A\subseteq X$ with $|A|=i$, 
and every maximal chain $\overline{A}$ in $A$ there
are exactly $(d-i)!$ maximal chains $\overline{X}$ in $X$
extending $\overline{A}$. Therefore, (\ref{eq:tauG}) is equal to
$$
\sum_{i=0}^d\sum_{A\subseteq X;|A|=i}\frac{1}{i!}\sum_{\underline{A}}
{\mathcal P}(r(\underline{A}))q^{\rk(A)}t^{|A|}=
\sum_{A\subseteq X}{\mathcal P}[{\bf X}\mid_A]q^{\rk(A)}t^{|A|}=
{\mathcal H}[{\bf X}](q,t).
$$

\end{proof}

\begin{corollary}
The quasi-symmetric function ${\mathcal F}[{\bf X}]$ specializes to ${\mathcal P}[{\bf X}]$
for matroids ${\bf X}$.
\end{corollary}
\begin{proof}
We define $\xi:\QSym_2\to \Sym$ by
$$
\xi(Q_{\alpha})=t^{\ell(\alpha)}\tau(Q_{\alpha})(1,t^{-1})\mid_{t=0}.
$$
One easily verifies that $\xi$ is a homomorphism of algebras,
and 
$$\xi({\mathcal G}[{\bf X}])={\mathcal H}[{\bf X}](1,t^{-1})t^{|X|}\mid_{t=0}={\mathcal P}[{\bf X}].
$$
for every matroid ${\bf X}=(X,\rk)$. 
Since $Q_{(1)}-Q_{(0)}$ lies in the kernel of $\xi$, 
$\xi$ factors through $\theta:\QSym_2\to \QSym\cong \QSym_2/(Q_{(1)}-Q_{(0)})$,
say $\xi=\eta\circ \theta$. Then we have
$$
{\mathcal P}[{\bf X}]=\xi({\mathcal G}[{\bf X}])=\eta(\theta({\mathcal G}[{\bf X})]))=
\eta({\mathcal F}[{\bf X}]).
$$
\end{proof}

\subsection{Speyer's invariant}
For a matroid ${\bf X}$ 
David Speyer defined an interesting polynomial $g_{\bf X}(t)$.
It has the multiplicative property ($g_{{\bf X_1}\oplus {\bf X_2}}(t)=g_{\bf X_1}(t)g_{\bf X_2}(t)$),
it is invariant under matroid-duality and has various other nice properties.
\begin{conjecture}\label{conjSpeyer}
The invariant ${\mathcal G}$ specializes to Speyer's invariant.
\end{conjecture}

\section{Polymatroid base polytopes}\label{sec:8}
\subsection{The valuative property of ${\mathcal G}$}
We will denote $\{1,2,\dots,n\}$ by $\underline{n}$.
For a polymatroid ${\bf X}=(\underline{n},\rk)$ we define its base polytope 
$\Poly(\rk)=\Poly_X(\rk)\subset\R^{n}$ by
$$
\textstyle \Poly(\rk)=\{v\in \R^n\mid
\sum^n_{i=1}v_i=\rk(\underline{n})\mbox{ and }
\forall A\subseteq \underline{n},\ \sum_{i\in A}v_i \leq \rk(A)\}.
$$
The $i$-th basis vector is denoted by $e_i$.
\begin{theorem}[see~\cite{HeHi}]
A compact convex polytope in $\R^n$ is the base polytope of a polymatroid if and only if
every vertice of the polytope has nonnegative integer coordinates, and every edge
is parallel to $e_j-e_k$ for some $j\neq k$.
\end{theorem}

For a compact convex polytope $\Pi\subset \R^n$,  its characteristic 
function $[\Pi]:\R^n\to \R$ is defined by
$$
[\Pi](x)=\left\{
\begin{array}{ll}
1 & \mbox{if $x\in\Pi$;}\\
0 &\mbox{if $x\not\in \Pi$.}
\end{array}
\right.
$$
Let ${\mathcal K}(\R^n)$ be the $\R$-vector space spanned by all $[\Pi]$
where $\Pi$ is a compact convex polytope. The {\em Euler
characteristic} is a linear function  $\chi:{\mathcal K}(\R^n)\to \R$
such that $\chi([\Pi])=1$ for every compact convex polytope $\Pi$
(see \cite[Theorem 7.4]{Barvinok} where $\chi$ is defined
for the slightly larger algebra of closed convex sets).

\begin{definition}
Suppose that $V$ is a $\Q$-vector space.
A $\Q$-linear map $f:\PolyMat\to V$ is called {\em valuative}
if it has the following property. For a finite set $X$  and
polymatroids ${\bf X}=(X,\rk_i)$, 
$i=1,2,\dots,r$ and rational numbers $a_1,\dots,a_r\in \Q$ 
such that
$$
\sum_{i=1}^r a_i [\Poly({\rk_i})]=0
$$
we have that
$$\sum_{i=1}^r a_if[{\bf X}_i]=0.
$$
Moreover, let us call $f$ {\em additive} if it is valuative
and $f([{\bf X}])=0$ whenever the polymatroid base
polytope $\Poly(\rk)$ of ${\bf X}=(X,\rk)$ has dimension $<n-1$.
\end{definition}
\begin{theorem}\label{theo:valuative}
$$
{\mathcal G}:\PolyMat\to \QSym
$$
is  valuative.
\end{theorem}
The proof of the theorem is in the next subsection.
\begin{corollary}
Since ${\mathcal G}$ specializes to ${\mathcal H}$ and ${\mathcal P}$,
these invariants are valuative as well.
\end{corollary}

A {\em polymatroid base decomposition} is a decomposition
\begin{equation}
\Poly(\rk)=\bigcup_{i=1}^r \Poly(\rk_i)
\end{equation}
such that
$$
\Poly(\rk_i)\cap \Poly(\rk_j)
$$
is a common face of $\Poly(\rk_i)$ and $\Poly(\rk_j)$ for $i\neq j$.
Let us call such a decomposition {\em proper} if $r>1$ and
$\Poly(\rk_i)\not\subseteq \Poly(\rk_j)$ for all $i\neq j$.
The polytope $\Poly(\rk)$ is called {\em indecomposable} if
it does not have a proper decomposition.
For a fixed base field $K$, a polymatroid is called {\em rigid} if it has
only finitely many realizations over $K$ as a subspace arrangement up to isomorphism.
The work of Lafforgue implies that a realizable matroid is {\em rigid} if and only if
its matroid base polytope is indecomposable (see~\cite{Lafforgue,Lafforgue2}).
It is therefore of interest to know whether a given matroid polytope is
indecomposable. Valuative and additive invariants can be useful
to determine whether a matroid polytope is decomposable.
For a valuative invariant $f$, we have, by the inclusion-exclusion principle
$$
f(\rk)=\sum_{k=1}^r (-1)^{k-1}\sum_{i_1<i_2<\cdots<i_k}f(\rk_{i_1,i_2,\dots,i_k})
$$
where $\rk_{i_1,\dots,i_k}$ is the rank function whose polymatroid polytope
is
$$\Poly(\rk_{i_1})\cap \cdots\cap \Poly(\rk_{i_k}).
$$
If $f$ is additive, then we have
$$
f(\rk)=\sum_{i=1}^r f(\rk_i).
$$
Additive invariants
can also be constructed from the Billera-Jia-Reiner quasi-symmetric function (see~\cite{BJR}).
\begin{conjecture}\label{conjUniversal}
Is ${\mathcal G}$ universal with respect to the valuative property?
I.e., is it true that for every 
$\Q$-linear valuative map  $f:\PolyMat\to V$
there exists a $\Q$-linear map $\psi:\QSym\to V$ such that $\psi\circ {\mathcal G}=f$?
\end{conjecture}

\subsection{The proof of Theorem~\ref{theo:valuative}}
The basis vectors of $\R^n$ are denoted by
$e_1,\dots,e_n$.
Let $\Delta$ be the $(n-2)$-dimensional simplex
spanned by $e_1-e_2,e_2-e_3,\dots,e_{n-1}-e_{n}$.
\begin{lemma}\label{lem:polyrank}
Choose $\varepsilon$ such that $0<\varepsilon<1$.
For $v\in \Z^n$, and a rank function $\rk:\Power(X)\to \R$, 
the following  statements are equivalent.
\begin{enumerate}
\item $\sum_{i=1}^s v_i=\rk(\underline{s})$ for $s=1,2,\dots,n$;
\item $v\in \Poly(\rk)$, and $v+\varepsilon(e_j-e_k)\not\in \Poly(\rk)$ for all $j<k$;
\item $(v+\varepsilon\Delta)\cap \Poly(\rk)=\emptyset$ and  $v\in \Poly(\rk)$.
\end{enumerate}
\end{lemma}
\begin{proof}
$(1)\Rightarrow (2)$:
Suppose that (1) holds.
Suppose that $A=\{i_1,\dots,i_s\}$ with $i_1<\cdots<i_s$.
Then we have
\begin{equation}\label{eq:rks}
\rk(\{i_1,\dots,i_t\})-\rk(\{i_1,i_2,\dots,i_{t-1}\})\leq
\rk(\{1,2,\dots,i_t\})-\rk(\{1,2,\dots,i_t-1\})=v_{i_t}
\end{equation}
by the submodular property of the rank function.

Summing (\ref{eq:rks}) for $t=1,2,\dots,s$ gives
$$
\rk(\{i_1,\dots,i_s\})\leq v_{i_1}+\cdots+v_{i_s}=\sum_{i\in A}v_i.
$$
This implies that $v\in \Poly(\rk)$.
If $j<k$ and $w=v+\varepsilon(e_j-e_k)$, then we have 
$$
\sum_{i=1}^jw_i=\sum_{i=1}^j v_i+\varepsilon=\rk(\underline{j})+\varepsilon>\rk(\underline{j}),
$$
so $w\not\in \Poly(\rk)$.
This proves that (2) holds.

$(2)\Rightarrow(1)$: Conversely, assume that (2) holds. A subset $S\subseteq \underline{n}$ 
is called {\em tight}
if $\sum_{i\in S}v_i=\rk(S)$. Clearly, $\underline{n}$ and $\emptyset$ are tight.
If $S,T$ are tight, then
\begin{multline}
\rk(S\cup T)+\rk(S\cap T)\leq \rk(S)+\rk(T)=
\sum_{i\in S}v_i+\sum_{i\in T}v_i=\\
=\sum_{i\in S\cap T}v_i+\sum_{i\in S\cup T}v_i\leq
\rk(S\cap T)+\rk(S\cup T),
\end{multline}
so all inequalities are equalities, and $S\cup T$ and $S\cap T$ are tight as well.

Suppose that $j<k$ and set $w=v+\varepsilon(e_j-e_k)$. 
Because $g\not\in \Poly(\rk)$, there exists a set $A_{j,k}$
such that
$$
\sum_{i\in A_{j,k}} w_i>\rk(A_{j,k}).
$$
Since 
$$
\sum_{i\in A_{j,k}}v_i\leq \rk(A_{j,k}),
$$
 we must have $j\in A_{j,k}$ and $k\not\in A_{j,k}$.
We obtain
$$
\rk(A_{j,k})\geq \sum_{i\in A_{j,k}}v_i= \sum_{i\in A_{j,k}} w_i-\varepsilon>\rk(A_{j,k})-\varepsilon.
$$
Because $v$ is an integer vector, the first inequality is an equality and $A_{j,k}$ is tight.
To prove (1) we need to show that $\underline{i}$ is tight
for $i=0,1,\dots,n$. We do this by induction on $i$, the case $i=0$
being trivial.
Suppose that $i>0$ and $\underline{i-1}$ is tight.
Then $\underline{i-1}\cup A_{i,k}$ is tight for  $k=i+1,\dots,n$.
We have
$$
\underline{i}=\bigcap_{k=i+1}^n (\underline{i-1}\cup A_{i,k})
$$
because $\underline{i}\subseteq \underline{i-1}\cup A_{i,k}$ for all $i$, and
$k\not \in \underline{i-1}\cup A_{i,k}$. Hence $\underline{i}$ is tight.

$(3)\Rightarrow (2)$: This implication is clear because $(e_j-e_k)\in \Delta$
for all $j<k$.

$(2)\Rightarrow (3)$: Suppose $v\in \Poly(\rk)$ and $v+\varepsilon (e_j-e_k)\not\in \Poly(\rk)$ for all $j<k$.
Suppose that $v+\delta (e_j-e_k)\in \Poly(\rk)$ for some $j,k$ with $j<k$ and  $\delta>0$.
Set $z:=e_j-e_k$.
If the inequality
\begin{equation}\label{eq:vis}
\sum_{i\in A}v_i\leq \rk(A).
\end{equation}
is an equality, then
$$
\rk(A)+\delta\sum_{i\in A}z_i=\sum_{i\in A}(v_i+\delta z_i)\leq \rk(A)
$$
because $v+\delta z\in \Poly(\rk)$. So we obtain
$$
\sum_{i\in A}z_i\leq 0,
$$
Therefore, we have
$$
\sum_{i\in A} (v_i+\varepsilon z_i)\leq \rk(A).
$$

If (\ref{eq:vis}) it is not tight, then
$$
\sum_{i\in A}v_i\leq \rk(A)-1
$$
and
$$
\sum_{i\in A}(v_i+\varepsilon z_i)\leq \rk(A)-1+\varepsilon\sum_{i\in A}z_i\leq
\rk(A)-1+\varepsilon \leq \rk(A).
$$
So we conclude that
$$
\sum_{i\in A}(v_i+\varepsilon z_i)\leq \rk(A)
$$
for all subsets $A\subseteq \underline{n}$.
So 
 $v+\varepsilon z\in \Poly(\rk)$, but this contradicts our assumptions.
 We conclude that $v+\delta (e_j-e_k)\not\in \Poly(\rk)$ for every $j<k$ and every $\delta>0$.
 
 Suppose that $v$ lies in the interior of a face of positive dimension of $\Poly(\rk)$. This
 face is parallel to $e_j-e_k$ for some $j<k$. This means that there exists a $\delta>0$
 such that $v+\delta(e_j-e_k),v-\delta(e_j-e_k)\in \Poly(\rk)$ for some $\delta>0$.
 This gives a contradiction, therefore $v$ must be a vertex of the polytope $\Poly(\rk)$.
 Let $v_1,v_2,\dots,v_r$ be other vertices of $\Poly(\rk)$ such that the
 edges of $\Poly(\rk)$ meeting at $v$ are $vv_1,vv_2,\dots,vv_r$.
 For every $v_i$, $v-v_i$ is a positive multiple of $e_k-e_j$ for some $j<k$.
 This means that $\Poly(\rk)$ is contained in cone
 $$
 C:=v+\R_{\geq 0}(e_2-e_1)+\R_{\geq 0}(e_3-e_2)+\cdots+\R_{\geq 0}(e_{n}-e_{n-1})
 $$
 where $\R_{\geq 0}$ denotes the nonnegative real numbers. We conclude that
 $$
 (v+\varepsilon \Delta)\cap \Poly(\rk)\subseteq (v+\varepsilon \Delta)\cap C=\emptyset.
 $$
 So (3) follows.
 \end{proof}
For $v\in \Z^n$, define a valuation $\mu_v:{\mathcal K}(\R)\to \R$ by 
$$\mu_v(h)=h(v)-\lim_{\varepsilon\downarrow 0}\chi([v+\varepsilon\Delta]\cdot h)$$
Let
$$
r=(r_1,r_2,\dots,r_n)
$$
where $r_i=\rk(\underline{i})-\rk(\underline{i-1})$ for all $i$.
\begin{corollary}
We have
$$
\mu_v([Poly(\rk)])=\left\{
\begin{array}{ll}
1 & \mbox{if $v=r$}\\
0 & \mbox{otherwise.}
\end{array}\right.
$$
\end{corollary}
\begin{proof}
Suppose that $v=r$.
By Lemma~\ref{lem:polyrank}, we have
$v\in \Poly(\rk)$ and $(v+\varepsilon\Delta)\cap \Poly(\rk)=\emptyset$.
Therefore, we get
$$\chi([v+\varepsilon \Delta]\cdot [\Poly(\rk)])=\chi([(v+\varepsilon\Delta)\cap \Poly(\rk)])=
\chi([\emptyset])=\chi(0)=0
$$
and $[\Poly(\rk)](v)=1$, so $\mu_v([\Poly(\rk)])=1$.

Suppose that $v\neq r$.
Assume that  $v\not\in \Poly(\rk)$. Since $\Poly(\rk)$ is closed,
there exists a $\delta>0$ such that
$$
(v+(\varepsilon\Delta))\cap \Poly(\rk)
$$
for all $\varepsilon$ with $0<\varepsilon<\delta$. 
This implies that $\mu_v([\Poly(\rk)])=0$.

Suppose that $v\in \Poly(\rk)$. Then 
 $(v+\varepsilon\Delta)\cap \Poly(\rk)$ is a closed nonempty convex polytope.
Hence we have
$$
\chi([v+\varepsilon\Delta]\cdot [\Poly(\rk)])=1.
$$
Therefore, we conclude that $\mu_v([\Poly(\rk)])=1-1=0$.
\end{proof}

\begin{proof}[Proof of Theorem~\ref{theo:valuative}]
The symmetric group $\Sigma_n$ acts on $\R^n$ by permuting the coordinates.
Define 
$$\mu_v^\sigma(h)=\mu_v(h\circ \sigma)$$ 
for every $\sigma\in \Sigma_n$ and every $h\in {\mathcal K}(\R)$.
We have that
\begin{multline}
\mu_v^\sigma([\Poly(\rk)])=\mu_v([\Poly(\rk\circ \sigma)])=\\
=\left\{
\begin{array}{ll}
1 & \mbox{if $v_i=\rk(\{\sigma(1),\dots,\sigma(i)\})-\rk(\{\sigma(1),\dots,\sigma(i-1)\})$ for all $i$;}\\
0 & \mbox{otherwise.}\end{array}\right.
\end{multline}
Define
$$
M_v=\sum_{\sigma\in \Sigma_n} \mu_v^\sigma.
$$
From the definition of ${\mathcal G}$ follows that 
$$
{\mathcal G}[{\bf X}]=\sum_{v}M_v([\Poly(\rk)])U_v.
$$
From  the linearity of $M_v$ and ${\mathcal G}$ it follows that
$$
\sum_{i}a_i{\mathcal G}[(\{1,\dots,n\},\rk_i)]=0
$$
whenever
$$
\sum_{i} a_i [\Poly(\rk_i)]=0.
$$
This completes the proof of the theorem.

\end{proof}
\section{Future directions}
For a polymatroid ${\bf X}$ we defined symmetric functions ${\mathcal P}[{\bf X}]$ and ${\mathcal H}[{\bf X}]$. In the case where the polymatroid comes from
a subspace arrangement,  we gave interpretations of the coefficients
of these symmetric functions
in terms of the Hilbert series and the minimal free resolution of 
the associated product ideal, and in terms of the polarized Schur functor.
We hope for simililar interpretations and nonnegativity results in the case
where the polymatroid is not realizable
(Conjecture~\ref{conNonneg}).
 We also defined a quasi-symmetric function ${\mathcal G}[{\bf X}]$.
This invariant has many interesting properties, and it specializes to ${\mathcal P}[{\bf X}]$, ${\mathcal H}[{\bf X}]$
and to the Billera-Jia-Reiner quasi-symmetric function ${\mathcal F}[{\bf X}]$. 
We would like to know whether ${\mathcal G}[{\bf X}]$ specializes to Speyer's invariant 
in~\cite{Speyer} (Conjecture~\ref{conjSpeyer}).
The invariant ${\mathcal G}$ behaves
valuatively with respect to (poly-)matroid base polytope decompositions. We wonder whether 
${\mathcal G}$ is universal with this property (Conjecture~\ref{conjUniversal}).

\end{document}